\documentclass[reqno, 12pt, oneside]{amsart}
\usepackage{amsmath, mathrsfs, url}
\usepackage[a4paper]{geometry}
\usepackage{enumitem}

\usepackage{bbm}

\allowdisplaybreaks

\newtheorem{thm}{Theorem}

\newtheorem{cor}{Corollary}

\newtheorem{lem}{Lemma}
\newtheorem{prop}{Proposition}
\newtheorem{rem}{Remark}

\newtheorem{assumption}{Assumption}

\title[Milstein-type scheme for MV-SDE]{On Explicit Milstein-type Scheme for Mckean-Vlasov Stochastic Differential Equations with Super-linear Drift Coefficient}

\author{Chaman Kumar}
\address{Department of Mathematics\\ Indian Institute of Technology Roorkee\\ Roorkee, Uttarakhand, India}
\email{chaman.kumar@ma.iitr.ac.in}

\author{Neelima}
\address{Department of Mathematics\\ Ramjas College\\ University of Delhi \\ Delhi, India}
\email{neelima\_maths@ramjas.du.ac.in}

\thanks{}

\begin{document}
\maketitle
\begin{abstract}
We develop an explicit Milstein-type scheme for McKean-Vlasov stochastic differential equations using the notion of derivative with respect to measure introduced by Lions and discussed in \cite{cardaliaguet2013}. The drift coefficient is allowed to grow super-linearly in the space variable. Further, both drift and diffusion coefficients are assumed to be only once differentiable in variables corresponding  to space and measure. The rate of strong convergence is shown to be equal to $1.0$ without using It\^o's formula for functions depending on measure. The challenges arising due to the dependence of coefficients on measure are tackled and our findings are consistent with the analogous results for  stochastic differential equations. 
\end{abstract}

\section{Introduction}
Let $(\Omega,\mathcal{F}, \{\mathcal{F}_t\}_{\{t\ge 0\}},P)$ be a filtered probability space satisfying the usual conditions.  
Assume that $\{W_t\}_{\{t\geq 0\}}$ is an $m$-dimensional Brownian motion. Consider the following $d$-dimensional McKean-Vlasov stochastic differential equation (MV-SDE),
\begin{align}
X_t=X_0+\int_0^t b(X_s, \mu_s^X) ds + \int_0^t \sigma(X_s,\mu_s^X)dW_s  \label{eq:MV_SDE:intro}
\end{align} 
almost surely for any $t\in[0,T]$ where $\mu_t^X$ denotes the law of the random variable $X_t$.  
When the law $\mu_t^X$ is known, MV-SDEs reduce to SDEs with added dependency on time variable. MV-SDEs are widely used in physics, biology and neural activities, see for example, \cite{baladron2012, bolley2011, bossy2015, dreyer2011, guhlke2018}. 
The main purpose of this article is to study Milstein-type numerical approximation of MV-SDEs \eqref{eq:MV_SDE:intro} in strong sense when the drift coefficient is allowed to grow super-linearly in space variable.  
There is a significant interest in the strong approximation of SDEs due to its importance in Multilevel Monte Carlo path simulations for SDEs, see \cite{giles2008}.   
It is well known that the classical Euler scheme for SDE with super-linear coefficients diverges in finite time, see for example \cite{hutzenthaler2010} and hence such divergence can obviously be observed in the case of MV-SDEs. 
The numerical approximation of SDEs in strong sense are well understood in the literature for global and non-global Lipschitz coefficients, see for example \cite{dareiotis2016,  hutzenthaler2015, hutzenthaler2012, kumar2017b,  kumar2019, kumar2017a, sabanis2013, sabanis2016, tretyakov2013,  wang2013} and references therein.  
 Recently, authors in  \cite{reis2019a} developed an explicit tamed Euler scheme and an implicit Euler scheme for simulating MV-SDEs when the drift coefficient satisfies non-global Lipschitz condition (and hence can grow super-linearly) and the diffusion coefficient satisfies global Lipschitz condition in space variable. We develop an explicit Milstein-type scheme for MV-SDEs \eqref{eq:MV_SDE:intro} and study its strong convergence. 
The drift coefficient is assumed to satisfy polynomial Lipschitz condition \textit{i.e.}, it is allowed to grow super-linearly and the diffusion coefficient satisfies global Lipschitz condition in the space variable. 
For the variable corresponding to measure,  coefficients satisfy global Lipschitz condition and are bounded in Wasserstein metric. 
Moreover, derivatives in space variable of drift and diffusion coefficients are respectively assumed to be polynomial Lipschitz and Lipschitz. 
Also, the derivative  of drift and diffusion coefficients with respect to measures satisfy Lipschitz condition and are bounded in Wasserstein metric. 
Further, the rate of strong convergence in $L^2$-norm is shown to be equal to $1$ without using It\^o's formula for functions depending on measure, see \cite{chassagneux}, which is consistent with the analogous result available in literature for SDEs.  
Novel techniques have been developed to tackle challenges arising due to the presence of the law $\mu_t$ in the coefficients.  
The following  mean-field stochastic Ginzburg Landau equation fits well in our framework, 
\begin{align*}
X_t=X_0+\int_0^t \big(\frac{\alpha^2}{2} X_s-X_s^3 + cE[X_s]\big)ds+\int_0^t \alpha X_s dW_s
\end{align*} 
almost surely for any $t\in[0,T]$. 
This equation has been investigated in \cite{reis2019a} and its variant (without mean-field term) in \cite{tien2013}.  
We also remark that the technique developed in this article can be used to investigate the higher order numerical approximations of MV-SDEs. 
To the best of author's knowledge, this is the first paper dealing with Milstein-type scheme for MV-SDEs using the notion of derivatives with respect to measures introduced by Lions in his lectures at the Coll\`ege de France and reproduced  in \cite{cardaliaguet2013}.

\subsection{Notations}
 We now introduce the notations used in this article. 
 The notation $\langle \cdot, \cdot \rangle$ stands for the inner product in $\mathbb{R}^d$.  
  We use the same notation $|\cdot|$ for both  Euclidean and  Hilbert-Schmidt norms and its meaning should be clear from the context. 
   Also, $\sigma x$ denotes the usual matrix multiplication of  $\sigma \in\mathbb{R}^{d \times m}$ and  $x\in\mathbb{R}^d$.
   With a slight abuse of notation,   $b^{(l)}$ and $\sigma^{(l)}$ are used to denote  $l$-th element of  $b\in\mathbb{R}^d$ and $l$-th column vector of $\sigma \in\mathbb{R}^{d \times m}$ respectively  which is clear from the context  and should not cause any confusion in reader's mind. 
  Further, $\sigma^{(k,l)}$ stands for $(k,l)$-th element of $\sigma\in\mathbb{R}^{d\times m}$. 
	For a function $f:\mathbb{R}^d \to \mathbb{R}$, $\partial_x f$ stands for gradient of $f$. 
	$\lfloor \cdot \rfloor$ stands for the floor function.
The symbol $\delta_x(\cdot)$ denotes the Dirac measure  at point $x\in\mathbb{R}^d$.  
Moreover, $\mathcal{P}_2(\mathbb{R}^d)$ denotes the space of probability measures $\mu$ on the measurable space $(\mathbb{R}^d, \mathcal{B}(\mathbb{R}^d))$ satisfying 
$
\int_{\mathbb{R}^d}|x|^2\mu(dx)<\infty. 
$ 
Then, $\mathcal{P}_2(\mathbb{R}^d)$ is a Polish space under the $\mathcal{L}^2$-Wasserstein metric given by 
\begin{align*}
\mathcal{W}_2(\mu_1,\mu_2):=\inf_{\pi\in\Pi(\mu_1,\mu_2)}\Big(\int_{\mathbb{R}^d}\int_{\mathbb{R}^d}|x-y|^2\pi(dx,dy) \Big)^{1/2}
\end{align*} 
where $\Pi(\mu_1,\mu_2)$ is the set of all couplings of $\mu_1, \mu_2\in\mathcal{P}_2(\mathbb{R}^d)$. 
Throughout this article, $K$ stands for a generic constant which may vary from place to place. 

\subsection{Differentiability of functions of  measures}
	There are many different notions for differentiating functions of measures, see for example \cite{ambrosio2008, villani2009}.  
	In this article, we use the notion of differentiability introduced by Lions in his lectures at the Coll\`ege de France which has been reproduced in \cite{cardaliaguet2013}.
We give a brief description of the concept of measure derivative for functions defined on the Wasserstein space $\mathcal{P}_2(\mathbb{R}^d)$.  
	A function $f:\mathcal{P}_2(\mathbb{R}^d)\to\mathbb{R}$ is said to be differentiable at  $\nu_0\in\mathcal{P}_2(\mathbb{R}^d)$ if there exists an atomless, Polish probability space $(\tilde{\Omega}, \tilde{\mathcal{F}}, \tilde{P})$   and a random variable $Y_0\in\mathcal{L}^2(\tilde{\Omega};\mathbb{R}^d)$ such that its law $L_{Y_0}:=\tilde{P}\circ Y_0^{-1}=\nu_0$ and the function $F:\mathcal{L}^2(\tilde{\Omega};\mathbb{R}^d)\to \mathbb{R}$ defined by $F(Z):=f(L_Z)$ has Fr\'echet derivative at $Y_0\in\mathcal{L}^2(\tilde{\Omega};\mathbb{R}^d)$, which we denote by  $F'[Y_0]$ . 
	The function $F$ is called the ``extension"  of $f$. 
	Further, $f$	is said to be of class $C^1$ if its ``extension" $F$ is of class $C^1$.  
	 Since $F'[Y_0]:\mathcal{L}^{2}(\tilde{\Omega};\mathbb{R}^d) \to \mathbb{R}$ is a bounded linear operator,  by Riesz representation theorem, there exists an element $DF(Y_0)\in \mathcal{L}^2(\tilde{\Omega};\mathbb{R}^d)$ such that $F'[Y_0](Z)=\tilde{E} \langle DF[Y_0], Z\rangle$  for all $Z\in\mathcal{L}^2 (\tilde{\Omega};\mathbb{R}^d)$.  
	By Theorem 6.5 (structure of the gradient) in \cite{cardaliaguet2013}, if $f$ is of class $C^1$, then there exists a function $\partial_\mu f(\nu_0): \mathbb{R}^d\to\mathbb{R}^d$ satisfying $
\int_{\mathbb{R}^d} |\partial_\mu f(\nu_0)(x)|^2 \nu_0(dx) <\infty
$
such that $DF(Y_0)=\partial_\mu f(\nu_0)(Y_0)$. 
	Also, $\partial_\mu f(\nu_0)$ is independent of choice of the probability space $(\tilde{\Omega}, \tilde{\mathcal{F}}, \tilde{P})$ and the random variable $Y_0$.  
	The function $\partial_\mu(\nu_0)$  is called the \textit{Lions' derivative} of $f$ at $\nu_0=L_{Y_0}$. 
	Moreover, $\partial_\mu f : \mathcal{P}_2(\mathbb{R}^d)\times \mathbb{R}^d\to\mathbb{R}^d$ is defined as $\partial_\mu f(\nu, z)=\partial_\mu f(\nu)(z)$ for any $\nu \in\mathcal{P}_2(\mathbb{R}^d)$ and  $z\in\mathbb{R}^d$.

%

\section{Assumptions and Main results}
	Let $(\Omega,\{\mathcal{F}_t\}_{\{t\geq 0\}}, \mathcal{F}, P)$ be a filtered probability space satisfying the usual conditions, \textit{i.e.}, the probability space $(\Omega, \mathcal{F}, P)$  is complete, $\mathcal{F}_0$ contains all $P$ - null sets of $\mathcal{F}$ and filtration is right continuous.
	Let $\{W_t\}_{\{t\geq 0\}}$ be an $m$-dimensional Brownian motion adapted to the filtration $\{\mathcal{F}_t\}_{\{t\geq 0\}}$. 
	Assume that $b:\mathbb{R}^d\times\mathcal{P}_2(\mathbb{R}^d) \to \mathbb{R}^{d}$ and $\sigma:\mathbb{R}^d\times\mathcal{P}_2(\mathbb{R}^d) \to \mathbb{R}^{d \times m}$ are  measurable functions.
	We consider the following McKean-Vlasov Stochastic Differential Equation (MV-SDE) defined on $(\Omega,\{\mathcal{F}_t\}_{\{t\geq 0\}}, \mathcal{F}, P)$,  
\begin{align} \label{eq:MVSDE}
X_t=X_0 +\int_0^t b(X_s, \mu_s^X) ds +\sum_{l=1}^m \int_0^t \sigma^{(l)}(X_s,\mu_s^X)dW_s^{(l)} 
\end{align}
almost surely for any $t\in[0,T]$ where $\mu^X_s$ denotes the law of  $X_s$, \textit{i.e.} $\mu_s^X~:=P \circ X_s^{-1}$ for every $s\in[0,T]$ and $X_0$ stands for an  $\mathbb{R}^d$-valued and $\mathcal{F}_0$-measurable random variable.   

We make the following assumptions on the coefficients and the initial value. 
 \begin{assumption} \label{as:initial}
  $E|X_0|^p< \infty $ for a fixed constant $p\geq 2$. 
 \end{assumption}
 \begin{assumption} \label{as:b:lip}
 There exist   constants $L>0$  and $\rho>0$ such that, 
\begin{align*}
 \big\langle x-\bar{x}, b(x,\mu)- b(\bar{x},\mu)\big\rangle &\leq L|x-\bar{x}|^2,
 \\
 |b(x,\mu)- b(\bar{x},\bar{\mu})| &\leq L \big\{ (1+|x|+|\bar{x}|)^{\rho/2+1}|x-\bar{x}|+  \mathcal{W}_2(\mu,\bar{\mu}) \big\},
\\
 |b(0,\mu)| &\leq L,
\end{align*}
for all $x,\bar{x} \in\mathbb{R}^d$ and $\mu, \bar{\mu}\in\mathcal{P}_2(\mathbb{R}^d)$.  \end{assumption}
 \begin{assumption}  \label{as:sig:lip}
 There exists a constant $L>0$ such that,
 \begin{align*}
 |\sigma(x,\mu)- \sigma(\bar{x},\bar{\mu})| &\leq L  \big\{|x-\bar{x}|+  \mathcal{W}_2(\mu, \bar{\mu}) \big\}, 
 \\
 |\sigma(0,\mu)| &\leq L,
 \end{align*}
 for all $x, \bar{x} \in\mathbb{R}^d$ and $\mu,\bar{\mu} \in\mathcal{P}_2(\mathbb{R}^d)$.  
 \end{assumption}
\begin{assumption}  \label{as:b:der}
 There exists a constant $L>0$ such that for every $k\in~\{1,\ldots,d\}$, the derivatives $\partial_x b^{(k)}:\mathbb{R}^d\times \mathcal{P}_2(\mathbb{R}^d)\to\mathbb{R}^d$ satisfy, 
 \begin{align*}
| \partial_x b^{(k)}(x,\mu)-\partial_x b^{(k)}(\bar{x},\bar{\mu})| & \leq L\big\{(1+|x|+|\bar{x}|)^{\rho/2}|x-\bar{x}|+ \mathcal{W}_2(\mu,\bar{\mu})\big\},
\\
|\partial_x b^{(k)}(0,\mu)| & \leq L,
 \end{align*}
 and the measure derivatives $\partial_\mu {b^{(k)}}:\mathbb{R}^d \times \mathcal{P}_2{(\mathbb{R}^d)} \times \mathbb{R}^d \to \mathbb{R}^d$  satisfy,
 \begin{align*}
| \partial_\mu b^{(k)}(x,\mu, y)-\partial_\mu b^{(k)}(\bar{x}, \bar{\mu}, \bar{y})| & \leq L\big\{(1+|x|+|\bar{x}|)^{\rho/2+1}|x-\bar{x}|+ \mathcal{W}_2(\mu,\bar{\mu}) +  |y-\bar{y}| \big\},
\\
|\partial_\mu b^{(k)}(0,\mu,0)| &\leq L,
 \end{align*}
 for all $x, y,\bar{x}, \bar{y}\in\mathbb{R}^d$ and $\mu,\bar{\mu} \in\mathcal{P}_2(\mathbb{R}^d)$.  
 \end{assumption}
 \begin{assumption}  \label{as:sig:der}
 There exists a constant $L>0$ such that for every $k\in~\{1,\ldots,d\}$ and $l\in~\{1,\ldots,m\}$, the derivatives $\partial_x \sigma^{(k,l)}: \mathbb{R}^d\times \mathcal{P}_2(\mathbb{R}^d)\to\mathbb{R}^d$ satisfy,
 \begin{align*}
| \partial_x \sigma^{(k,l)}(x,\mu)-\partial_x \sigma^{(k,l)}(\bar{x},\bar{\mu})| & \leq L\big\{|x-\bar{x}| + \mathcal{W}_2(\mu,\bar{\mu}) \big\},
\\
|\partial_x \sigma^{(k,l)}(0,\mu)| &\leq L, 
 \end{align*}
 and the measure derivatives $\partial_\mu {\sigma^{(k,l)}}:\mathbb{R}^d \times \mathcal{P}_2{(\mathbb{R}^d)} \times \mathbb{R}^d \to \mathbb{R}^d$ satisfy,
 \begin{align*}
| \partial_\mu \sigma^{(k,l)}(x,\mu, y)-\partial_\mu \sigma^{(k,l)}(\bar{x}, \bar{\mu}, \bar{y})| & \leq L\big\{|x-\bar{x}|+ \mathcal{W}_2(\mu,\bar{\mu}) +  |y-\bar{y}| \big\},
\\
|\partial_\mu \sigma^{(k,l)}(0,\mu,0)| &\leq L, 
 \end{align*}
 for all $x, y,\bar{x}, \bar{y}\in\mathbb{R}^d$ and $\mu,\bar{\mu} \in\mathcal{P}_2(\mathbb{R}^d)$.  
 \end{assumption}

\subsection{Propagation of Chaos and Interacting Particle System}
For a fixed $N\in\mathbb{N}$, let $\{W^i\}_{i\in\{1,\ldots,N\}}$ be $N$ independent Brownian motions that are also independent of $W$. Consider $N$-dimensional system of interacting particles given by, 
\begin{align}
X_t^{i,N}=X_0^{i}+ \int_0^t b\big(X_s^{i,N}, \mu_s^{X,N}\big) ds + \sum_{l=1}^m \int_0^t \sigma^{(l)}\big(X_s^{i,N}, \mu_s^{X,N}\big) dW_s^{(l),i} \label{eq:interactint:particle}
\end{align}
almost surely for any $t\in[0,T]$ and $i\in\{1,\ldots,N\}$, where 
\begin{align*}
\mu_s^{X,N}:=\frac{1}{N} \sum_{j=1}^{N} \delta_{X_s^{j,N}} 
\end{align*}
for any $s\in[0,T]$. For the propagation of chaos result, consider the system of non-interacting particles given by, 
\begin{align}
X_t^{i}=X_0^{i}+ \int_0^t b\big(X_s^{i}, \mu_s^{X^i}\big) ds + \sum_{l=1}^m \int_0^t \sigma^{(l)}\big(X_s^{i}, \mu_s^{X^i}\big) dW_s^{(l),i} \label{eq:noninteracting:nonparticle}
\end{align}
almost surely for any $t\in[0,T]$ and $i\in\{1,\ldots,N\}$, where $\mu_s^{X^i}=\mu_s^{X}$ for every $i \in \{1,\ldots,N\}$ because $X^i$'s are independent. 

The proof of the following proposition can be found in \cite{reis2019a, reis2019b}.
\begin{prop} \label{prop:chaos}
Let Assumptions \ref{as:initial}, \ref{as:b:lip} and \ref{as:sig:lip} be satisfied. Then, there exists a unique solution to MV-SDE \eqref{eq:MVSDE} and the following holds, 
\begin{align*}
E\sup_{t\in[0,T]} \big|X_t\big|^p \leq K
\end{align*}
where  $K:=K(m,d,L,p,T, E|X_0|^p)>0$ is a constant.  Also, 
 \[
 \sup_{i \in \{1,\ldots,N\}}E\sup_{t\in[0,T]} \big|X_t^i-X_t^{i,N}\big|^2\leq K
\begin{cases}
 N^{-1/2} & \mbox{ if } d<4,
\\
  N^{-1/2} \ln(N) & \mbox{ if } d=4,
\\
  N^{-2/d}  & \mbox{ if } d>4,
\end{cases}
\]
where the constant $K>0$ does not depend on $N$. 
\end{prop}
 
\subsection{Explicit Milstein-type Scheme} 
For introducing Milstein-type scheme for MV-SDEs \eqref{eq:MVSDE}, we partition the interval $[0,T]$ into $n$ sub-intervals each of length $h=T/n$ and define $\kappa_n(s):=\lfloor ns \rfloor/ n$ for any $s\in[0,T]$. 
For any $x\in\mathbb{R}^d$ and $\mu\in\mathcal{P}_2(\mathbb{R}^d)$, define 
\begin{align} \label{eq:bn}
b_n\big(x, \mu\big) :=\frac{b\big(x, \mu\big) }{1+n^{-1}|x|^{\rho+2}}  
\end{align}
for every $n\in\mathbb{N}$. 
We propose the following explicit Milstein-type scheme for MV-SDE \eqref{eq:MVSDE}, 
\begin{align} \label{eq:scheme}
X_t^{i,N, n}=X_0^{i}+ \int_0^t b_n\big(X_{\kappa_n(s)}^{i,N, n}, \mu_{\kappa_n(s)}^{X,N,n}\big) ds +\sum_{l=1}^m \int_0^t \tilde{\sigma}^{(l)}\big(s,X_{\kappa_n(s)}^{i,N,n}, \mu_{\kappa_n(s)}^{X,N,n}\big) dW_s^{(l),i}
\end{align}
almost surely for any $t\in[0,T]$ and $n,N\in\mathbb{N}$, where 
\begin{align*}
\mu_{\kappa_n(s)}^{X,N}:=\frac{1}{N} \sum_{j=1}^{N} \delta_{X_{\kappa_n(s)}^{j,N,n}}  
\end{align*}
for any $s\in[0,T]$ and $n,N\in\mathbb{N}$. Also,  for every $l\in\{1,\ldots,m\}$,  $\tilde{\sigma}^{(l)}$ is given by, 
\begin{align} \label{eq:sigma:tilde}
\tilde{\sigma}^{(l)} \big(s,X_{\kappa_n(s)}^{i,N,n}, \mu_{\kappa_n(s)}^{X,N,n}\big)&:=\sigma^{(l)} \big(X_{\kappa_n(s)}^{i,N,n}, \mu_{\kappa_n(s)}^{X,N,n}\big)+\Lambda_1^{(l)} \big(s,X_{\kappa_n(s)}^{i,N,n}, \mu_{\kappa_n(s)}^{X,N,n}\big) \notag
\\
&+\Lambda_2^{(l)} \big(s,X_{\kappa_n(s)}^{i,N,n}, \mu_{\kappa_n(s)}^{X,N,n}\big)
\end{align}
where $\Lambda_1^{(l)} \big(s,X_{\kappa_n(s)}^{i,N,n}, \mu_{\kappa_n(s)}^{X,N,n}\big)$ and $\Lambda_2^{(l)} \big(s,X_{\kappa_n(s)}^{i,N,n}, \mu_{\kappa_n(s)}^{X,N,n}\big)$  are the $l$-th column of $d\times m$-matrices whose $(k,l)$-th elements are respectively given  by,
\begin{align*}
\Lambda_1^{(k,l)}& \big(s,X_{\kappa_n(s)}^{i,N,n}, \mu_{\kappa_n(s)}^{X,N,n}\big)
\\
&:= \Big\langle \partial_x \sigma^{(k,l)} \big(X_{\kappa_n(s)}^{i,N,n}, \mu_{\kappa_n(s)}^{X,N,n}\big), \sum_{l_1=1}^m \int_{\kappa_n(s)}^s \sigma^{(l_1)} \big(X_{\kappa_n(r)}^{i,N,n}, \mu_{\kappa_n(r)}^{X,N,n}\big) dW_r^{(l_1),i} \Big\rangle
\\
& = \sum_{l_1=1}^m \int_{\kappa_n(s)}^s  \Big\langle \partial_x \sigma^{(k,l)} \big(X_{\kappa_n(s)}^{i,N,n}, \mu_{\kappa_n(s)}^{X,N,n}\big), \sigma^{(l_1)} \big(X_{\kappa_n(r)}^{i,N,n}, \mu_{\kappa_n(r)}^{X,N,n}\big) \Big\rangle dW_r^{(l_1),i}, 
\\
\Lambda_2^{(k,l)}  & \big(s,X_{\kappa_n(s)}^{i,N,n}, \mu_{\kappa_n(s)}^{X,N,n}\big)
\\
&:= \frac{1}{N}\sum_{j=1}^N  \Big\langle \partial_\mu \sigma^{(k,l)} \big( X_{\kappa_n(s)}^{i,N,n}, \mu_{\kappa_{n}(s)}^{X,N,n},  X_{\kappa_n(s)}^{j,N,n} \big), \sum_{l_1=1}^m \int_{\kappa_n(s)}^s \sigma^{(l_1)} \big(X_{\kappa_n(r)}^{j,N,n}, \mu_{\kappa_n(r)}^{X,N,n}\big) dW_r^{(l_1),j} \Big\rangle
\\
& =  \frac{1}{N}\sum_{j=1}^N  \sum_{l_1=1}^m \int_{\kappa_n(s)}^s  \Big\langle \partial_\mu \sigma^{(k,l)} \big( X_{\kappa_n(s)}^{i,N,n}, \mu_{\kappa_{n}(s)}^{X,N,n},  X_{\kappa_n(s)}^{j,N,n} \big), \sigma^{(l_1)} \big(X_{\kappa_n(r)}^{j,N,n}, \mu_{\kappa_n(r)}^{X,N,n}\big) \Big\rangle dW_r^{(l_1),j} 
\end{align*}
for any $s\in[0,T]$ and $n, N \in\mathbb{N}$ and for every $k\in\{1,\ldots,d\}$ and $l\in\{1,\ldots,m\}$. 
The following proposition and theorem are the main results of this article. 
\begin{prop} \label{prop:milstein}
Let Assumptions \ref{as:initial} to \ref{as:sig:der} be satisfied. Then, the Milstein-type scheme \eqref{eq:scheme} converges to the interacting particle system \eqref{eq:interactint:particle} with the rate of convergence given by, 
\begin{align*}
\sup_{i\in\{1,\ldots,N\}} \sup_{t\in[0,T]} E\big|X_t^{i,N}-X_t^{i,N, n}\big|^2 \leq K n^{-2}
\end{align*}
for any $n, N\in\mathbb{N}$ where the constant $K>0$ does not depend on $n,N\in\mathbb{N}$. 
\end{prop}
By combining Propositions \ref{prop:chaos} and \ref{prop:milstein}, we obtain the following theorem. 
\begin{thm}
Let Assumptions \ref{as:initial} to \ref{as:sig:der} be satisfied. Then, the Milstein-type scheme \eqref{eq:scheme} converges to the true solution of MV-SDE \eqref{eq:MVSDE} with the rate of convergence given by, 
\[
 \sup_{i \in \{1,\ldots,N\}} \sup_{t\in[0,T]}  E\big|X_t^i-X_t^{i,N,n}\big|^2\leq K
\begin{cases}
 N^{-1/2}+ n^{-2} & \mbox{ if } d<4,
\\
  N^{-1/2} \ln(N) + n^{-2}& \mbox{ if } d=4,
\\
  N^{-2/d} + n^{-2} & \mbox{ if } d>4,
\end{cases}
\]
for any $n, N\in\mathbb{N}$ where constant $K>0$ does not depend on $n, N\in\mathbb{N}$. 
\end{thm}
 We conclude this section by listing following remarks which are consequences of the assumptions mentioned above. 
 \begin{rem} \label{rem:b:growth}
 From  Assumptions \ref{as:b:lip} and \ref{as:b:der},
 \begin{align*}
 \langle x, b(x, \mu)  \rangle & \leq K(1+|x|)^2, 
 \\
 |b(x,\mu)| & \leq K(1+|x|)^{\rho/2+2},
 \\
 |\partial_x b^{(k)}(x,\mu)| & \leq K(1+|x|)^{\rho/2+1},
 \\
 |\partial_\mu b^{(k)}(x,\mu,y)| & \leq K \big\{(1+|x|)^{\rho/2+2}+ (1+|y|)\big\},
\end{align*}  
for any $x, y \in\mathbb{R}^d$, $\mu\in\mathcal{P}_2(\mathbb{R}^d)$ and  $k\in\{1,\ldots,d\}$. 
 \end{rem}
 \begin{rem} \label{rem:bn:growth}
 From Remark \ref{rem:b:growth} and equation \eqref{eq:bn}, 
 \begin{align*}
 |b_n(x,\mu)| & \leq K \min\big\{n^{1/2}(1+|x|),  (1+|x|)^{\rho/2+2}\big\}
 \end{align*}
 for all $x\in\mathbb{R}^d$, $\mu\in\mathcal{P}_2(\mathbb{R}^d)$ and $n\in\mathbb{N}$ where the constant $K>0$ does not depend on $n\in\mathbb{N}$. 
 \end{rem}
 \begin{rem} \label{rem:sigma:growth}
 From Assumptions \ref{as:sig:lip} and \ref{as:sig:der}, 
 \begin{align*}
 |\sigma(x,\mu)| & \leq K (1+|x|),
 \\
 |\partial_x \sigma^{(k,l)}(x,\mu)| & \leq K,
 \\
  |\partial_\mu \sigma^{(k,l)}(x,\mu,y)| & \leq K,
 \end{align*}
 for any $x,y\in\mathbb{R}^d$, $\mu\in\mathcal{P}_2(\mathbb{R}^d)$, $k\in\{1,\ldots,d\}$  and  $l\in\{1,\ldots,m\}$. 
 \end{rem}
\section{Moment Bounds}
Before establishing the moment bound of the Milstein-type scheme \eqref{eq:scheme} in Lemma~\ref{lem:moment:bound}, we first establish following lemmas and corollaries. 
\begin{lem} \label{lem:lambda_1}
Let Assumptions  \ref{as:sig:lip} and \ref{as:sig:der} be satisfied. Then, 
\begin{align*}
E|\Lambda_1^{(u,v)} \big(s,X_{\kappa_n(s)}^{i,N,n}, \mu_{\kappa_n(s)}^{X,N,n}\big)|^p \leq K n^{-\frac{p}{2}} E (1+|X_{\kappa_n(s)}^{i,N,n}|\big)^p
\end{align*}
for any $s\in[0,T]$, $i\in\{1,\cdots,N\}$, $n, \, N\in\mathbb{N}$, $u\in\{1,\ldots,d\}$ and $v\in\{1,\ldots,m\}$, where  constant $K>0$ does not depend on $n$ and $N$. 
\end{lem}
\begin{proof}
By  Cauchy-Schwarz inequality and  Burkholder-Gundy-Davis inequality,  
\begin{align*} 
E|& \Lambda_1^{(u,v)}  \big(s,X_{\kappa_n(s)}^{i,N,n}, \mu_{\kappa_n(s)}^{X,N,n}\big)|^p 
\\
&=E \Big| \sum_{l_1=1}^m \int_{\kappa_n(s)}^s \Big\langle \partial_x \sigma^{(u,v)} \big(X_{\kappa_n(s)}^{i,N,n}, \mu_{\kappa_n(s)}^{X,N,n}\big),  \sigma^{(l_1)} \big(X_{\kappa_n(r)}^{i,N,n}, \mu_{\kappa_n(r)}^{X,N,n}\big)  \Big\rangle dW_r^{(l_1),i} \Big|^p
\\
&\leq K n^{-\frac{p}{2}+1}  E  \sum_{l_1=1}^m \int_{\kappa_n(s)}^s \big| \partial_x \sigma^{(u,v)} \big(X_{\kappa_n(s)}^{i,N,n}, \mu_{\kappa_n(s)}^{X,N,n}\big)\big|^p \big| \sigma^{(l_1)} \big(X_{\kappa_n(r)}^{i,N,n}, \mu_{\kappa_n(r)}^{X,N,n}\big)  \big|^p dr 
\end{align*}
and then the application of Remark \ref{rem:sigma:growth} completes the proof. 
\end{proof}
\begin{lem} \label{lem:lambda2}
Let Assumptions \ref{as:sig:lip} and \ref{as:sig:der} be satisfied. Then, 
\begin{align*}
E|\Lambda_2^{(u,v)} \big(s,X_{\kappa_n(s)}^{i,N,n}, \mu_{\kappa_n(s)}^{X,N,n}\big)|^p \leq K n^{-\frac{p}{2}} \frac{1}{N}    \sum_{j=1}^N    E\big(1+|X_{\kappa_n(s)}^{j,N,n}|\big)^p  
\end{align*}
for any $s\in[0,T]$, $i\in\{1,\cdots,N\}$, $n, \, N\in\mathbb{N}$, $u\in\{1,\ldots,d\}$ and $v\in\{1,\ldots,m\}$, where  constant $K>0$ does not depend on $n$ and $N$.
\end{lem}
\begin{proof}
On using Cauchy-Schwarz inequality and Burkholder-Gundy-Davis inequality, 
\begin{align*}
E&|\Lambda_2^{(u,v)} \big(s,X_{\kappa_n(s)}^{i,N,n}, \mu_{\kappa_n(s)}^{X,N,n}\big)|^p
\\
& = E \Big|\frac{1}{N}\sum_{j=1}^N   \sum_{l_1=1}^m \int_{\kappa_n(s)}^s \Big\langle \partial_\mu \sigma^{(u,v)}(X_{\kappa_n(s)}^{i,N,n}, \mu_{\kappa_{n}(s)}^{X,N,n}, X_{\kappa_n(s)}^{j,N,n}), \sigma^{(l_1)} \big(X_{\kappa_n(r)}^{j,N,n}, \mu_{\kappa_n(r)}^{X,N,n}\big) \Big\rangle dW_r^{(l_1),j}  \Big|^p
\\
& \leq Kn^{-\frac{p}{2}+1}  E \frac{1}{N}\sum_{j=1}^N   \sum_{l_1=1}^m \int_{\kappa_n(s)}^s \big|\partial_\mu \sigma^{(u,v)}(X_{\kappa_n(s)}^{i,N,n}, \mu_{\kappa_{n}(s)}^{X,N,n}, X_{\kappa_n(s)}^{j,N,n})\big|^p \big| \sigma^{(l_1)} \big(X_{\kappa_n(r)}^{j,N,n}, \mu_{\kappa_n(r)}^{X,N,n}\big) \big|^p dr  
\end{align*}
and then the proof is completed by  using Remark \ref{rem:sigma:growth}. 
\end{proof}
As a consequence of Remark \ref{rem:sigma:growth}, Lemma \ref{lem:lambda_1} and Lemma \ref{lem:lambda2}, one obtains the following corollary. 
\begin{cor} \label{cor:sigma:tilde}
Let Assumptions \ref{as:sig:lip} and \ref{as:sig:der} be satisfied. Then,
\begin{align*}
E\big|\tilde{\sigma}  \big(s,X_{\kappa_n(s)}^{i,N,n}, & \mu_{\kappa_n(s)}^{X,N,n}\big)\big|^p \leq K  E \big(1+|X_{\kappa_n(s)}^{i,N,n}|\big)^p + K n^{-\frac{p}{2}} E \big(1+|X_{\kappa_n(s)}^{i,N,n}|\big)^p 
\\
& + K n^{-\frac{p}{2}} \frac{1}{N}    \sum_{j=1}^N    E\big(1+|X_{\kappa_n(s)}^{j,N,n}|\big)^p
\end{align*}
for any $s\in[0,T]$,  $i\in\{1,\cdots,N\}$ and $n, \, N\in\mathbb{N}$ where the  constant $K>0$ does not depend on $n$ and $N$. 
\end{cor}
\begin{lem} \label{lem:one:step}
Let Assumptions \ref{as:b:lip} to \ref{as:sig:der} be satisfied. Then,
\begin{align*}
E|X_s^{i,N, n}-X_{\kappa_n(s)}^{i,N, n}|^p \leq K n^{-\frac{p}{2}} E \big(1+|X_{\kappa_n(s)}^{i,N,n}|\big)^p + K n^{-\frac{p}{2}} \frac{1}{N}    \sum_{j=1}^N    E\big(1+|X_{\kappa_n(s)}^{j,N,n}|\big)^p
\end{align*}
for any $s\in[0,T]$,  $i\in\{1,\cdots,N\}$ and $n, \, N\in\mathbb{N}$ where the  constant $K>0$ does not depend on $n$ and $N$. 
\end{lem}
\begin{proof}
 From equation \eqref{eq:scheme}, one can get the following estimate, 
\begin{align*} 
E|& X_s^{i,N, n}-X_{\kappa_n(s)}^{i,N, n}|^p \leq  K E\Big|\int_{\kappa_n(s)}^s b_n\big(X_{\kappa_n(r)}^{i,N, n}, \mu_{\kappa_n(r)}^{X,N,n}\big) dr\Big|^p
\\
&  + KE \Big|\sum_{l=1}^m \int_{\kappa_n(s)}^s \tilde{\sigma}^{(l)}\big(r,X_{\kappa_n(r)}^{i,N,n}, \mu_{\kappa_n(r)}^{X,N,n}\big) dW_r^{(l),i} \Big|^p 
\end{align*}
and then the application of H\"older's inequality and Burkholder-Gundy-Davis inequality gives, 
\begin{align*} 
E|& X_s^{i,N, n}-X_{\kappa_n(s)}^{i,N, n}|^p \leq K n^{-p+1}E\int_{\kappa_n(s)}^s |b_n\big(X_{\kappa_n(r)}^{i,N, n}, \mu_{\kappa_n(r)}^{X,N,n}\big)|^p dr
\\
&   + Kn^{-\frac{p}{2}+1} E \sum_{l=1}^m \int_{\kappa_n(s)}^s |\tilde{\sigma}^{(l)}\big(r,X_{\kappa_n(r)}^{i,N,n}, \mu_{\kappa_n(r)}^{X,N,n}\big) |^p dr 
\end{align*} 
which on using Remark \ref{rem:bn:growth} and Corollary \ref{cor:sigma:tilde} completes the proof. 
\end{proof}
\begin{lem} \label{lem:moment:bound}
Let Assumptions  \ref{as:initial} to \ref{as:sig:der} be satisfied. Then,
\begin{align*}
\sup_{i\in\{1,\ldots,N\}} E &\sup_{t\in [0,T]}(1+|X_{t}^{i,N, n}|^2)^{p/2} \leq K
\end{align*}
for any $n,\, N\in\mathbb{N}$ where the constant $K>0$ does not depend on $n$ and $N$. 
\end{lem}
\begin{proof}
By the application of It\^o's formula, 
\begin{align*}
\big(1+&|X_{t}^{i,N, n}|^2\big)^{p/2} = \big(1+|X_{0}^{i}|^2\big)^{p/2} 
\\
& + p \int_0^t \big(1+|X_{s}^{i,N, n}|^2\big)^{p/2-1} \big\langle X_{s}^{i,N, n}, b_n\big(X_{\kappa_n(s)}^{i,N, n}, \mu_{\kappa_n(s)}^{X,N,n}\big)\big\rangle  ds
\\
& + p  \int_0^t \big(1+|X_{s}^{i,N, n}|^2\big)^{p/2-1} \big\langle X_{s}^{i,N, n},  \tilde{\sigma}\big(s, X_{\kappa_n(s)}^{i,N, n}, \mu_{\kappa_n(s)}^{X,N,n}\big) dW_s^{i} \big\rangle 
 \\
 & + \frac{p(p-2)}{2}  \int_0^t \big(1+|X_{s}^{i,N, n}|^2\big)^{p/2-2}  \big|\tilde{\sigma}^{*}\big(s, X_{\kappa_n(s)}^{i,N, n}, \mu_{\kappa_n(s)}^{X,N,n}\big)X_{s}^{i,N, n}\big|^2 ds
 \\
 & + \frac{p}{2}  \int_0^t \big(1+|X_{s}^{i,N, n}|^2\big)^{p/2-1}  \big|\tilde{\sigma}\big(s, X_{\kappa_n(s)}^{i,N, n}, \mu_{\kappa_n(s)}^{X,N,n}\big)\big|^2  ds
\end{align*}
almost surely, which on the application of Burkholder-Gundy-Davis inequality  and  Cauchy-Schwarz inequality  yields the following estimate, 
\begin{align*}
E  \sup_{t\in [0,u]}&\big(1+|X_{t}^{i,N, n}|^2\big)^{p/2} \leq E\big(1+|X_{0}^{i}|^2\big)^{p/2} 
\\
& + K E \sup_{t\in[0,u]} \Big|\int_0^t \big(1+|X_{s}^{i,N, n}|^2\big)^{p/2-1} \big\langle X_{\kappa_n(s)}^{i,N, n}, b_n\big(X_{\kappa_n(s)}^{i,N, n}, \mu_{\kappa_n(s)}^{X,N,n}\big) \big\rangle ds \Big|
\\
& + K E \sup_{t\in[0,u]} \Big| \int_0^t \big(1+|X_{s}^{i,N, n}|^2\big)^{p/2-1}  \big\langle X_{s}^{i,N, n}- X_{\kappa_n(s)}^{i,N, n} , b_n\big(X_{\kappa_n(s)}^{i,N, n}, \mu_{\kappa_n(s)}^{X,N,n}\big) \big\rangle ds \Big|
\\
& + K E\Big\{\int_0^u \big(1+|X_{s}^{i,N, n}|^2\big)^{p-2} | \tilde{\sigma}\big(s, X_{\kappa_n(s)}^{i,N, n}, \mu_{\kappa_n(s)}^{X,N,n}\big)|^2 ds \Big\}^\frac{1}{2}
 \\
 & + K  E \int_0^u \big(1+|X_{s}^{i,N, n}|^2\big)^{p/2-1}  \big|\tilde{\sigma}\big(s, X_{\kappa_n(s)}^{i,N, n}, \mu_{\kappa_n(s)}^{X,N,n}\big)\big|^2  ds
\end{align*}
for  any $i\in\{1,\ldots,N\}$, $n, N\in \mathbb{N}$,  and $u\in[0,T]$.  Also, one uses Remark \ref{rem:b:growth}, Cauchy-Schwarz inequality and Young's inequality to obtain the following estimate, 
\begin{align*}
E \sup_{t\in [0,u]}\big(1+&|X_{t}^{i,N, n}|^2\big)^{p/2} \leq E\big(1+|X_{0}^{i}|^2\big)^{p/2} + K  \int_0^u E \sup_{r\in[0,s]}\big(1+|X_{r}^{i,N, n}|^2\big)^{p/2}  ds
\\
& + K E \int_0^u \big(1+|X_{s}^{i,N, n}|^2\big)^{p/2-1} \big(1+ |X_{\kappa_n(s)}^{i,N, n}|^2\big)ds  
\\
& + K E \int_0^u  n^{p/4} \big|X_{s}^{i,N, n}- X_{\kappa_n(s)}^{i,N, n} \big|^{p/2} n^{-p/4} \big|b_n\big(X_{\kappa_n(s)}^{i,N, n}, \mu_{\kappa_n(s)}^{X,N,n}\big)\big|^{p/2} ds
\\
& + K E\Big\{\int_0^u  \big| \tilde{\sigma}\big(s, X_{\kappa_n(s)}^{i,N, n}, \mu_{\kappa_n(s)}^{X,N,n}\big)\big|^2 ds \Big\}^{p/2}
 \\
 & + K  E \int_0^u   \big|\tilde{\sigma}\big(s, X_{\kappa_n(s)}^{i,N, n}, \mu_{\kappa_n(s)}^{X,N,n}\big)\big|^p  ds
\end{align*}
and then due to H\"older's inequality  and Young's inequality, one gets 
\begin{align*}
E &\sup_{t\in [0,u]}\big(1+|X_{t}^{i,N, n}|^2\big)^{p/2} \leq E\big(1+|X_{0}^{i}|^2\big)^{p/2}  + K  \int_0^u E \sup_{r\in[0,s]}\big(1+|X_{r}^{i,N, n}|^2)^{p/2}  ds  
\\
& + K E \int_0^u  n^\frac{p}{2} |X_{s}^{i,N, n}- X_{\kappa_n(s)}^{i,N, n} |^p ds + K E \int_0^u   n^{-\frac{p}{2}}|b_n\big(X_{\kappa_n(s)}^{i,N, n}, \mu_{\kappa_n(s)}^{X,N,n}\big)|^p ds
 \\
 & + K  E \int_0^u   |\tilde{\sigma}\big(s, X_{\kappa_n(s)}^{i,N, n}, \mu_{\kappa_n(s)}^{X,N,n}\big)|^p  ds
\end{align*}
for any $i \in\{1,\ldots, N\}$, $n, N\in\mathbb{N}$  and $u\in[0,T]$. The application of Remark \ref{rem:b:growth}, Lemma \ref{lem:one:step} and Corollary \ref{cor:sigma:tilde} yields, 
\begin{align*}
\sup_{i\in\{1,\ldots,N\}} E &\sup_{t\in [0,u]}(1+|X_{t}^{i,N, n}|^2)^{p/2} \leq E(1+|X_{0}^2)^{p/2} 
\\
& + K  \int_0^u \sup_{i\in\{1,\ldots,N\}}E \sup_{r\in[0,s]}(1+|X_{r}^{i,N, n}|^2)^{p/2}  ds  
\end{align*}
for any  $n, N\in\mathbb{N}$  and $u\in[0,T]$. Finally, the proof is completed by using the Gronwall's inequality. 
\end{proof}

\section{Rate of Convergence}
In this section, we shall prove Proposition \ref{prop:milstein}. For this, we require some lemmas and corollaries which are shown below. Notice that as a consequence of Lemmas \ref{lem:lambda_1}, \ref{lem:lambda2} and  \ref{lem:moment:bound}, we obtain the following corollaries. 
\begin{cor} \label{cor:lambda_1}
Let Assumptions  \ref{as:initial} to \ref{as:sig:der} be satisfied. Then,
\begin{align*}
E|\Lambda_1 \big(s,X_{\kappa_n(s)}^{i,N,n}, \mu_{\kappa_n(s)}^{X,N,n}\big)|^p \leq K n^{-\frac{p}{2}} 
\end{align*}
for any $s\in[0,T]$,  $i\in\{1,\cdots,N\}$ and $n, \, N\in\mathbb{N}$ where the  constant $K>0$ does not depend on $n$ and $N$. 
\end{cor}
\begin{cor} \label{cor:lambda2}
Let Assumptions  \ref{as:initial} to \ref{as:sig:der} be satisfied. Then,
\begin{align*}
E|\Lambda_2 \big(s,X_{\kappa_n(s)}^{i,N,n}, \mu_{\kappa_n(s)}^{X,N,n}\big)|^p \leq K n^{-\frac{p}{2}}   
\end{align*}
for any $s\in[0,T]$,  $i\in\{1,\cdots,N\}$ and $n, \, N\in\mathbb{N}$ where the  constant $K>0$ does not depend on $n$ and $N$. 
\end{cor}
\begin{cor} \label{cor:sigma:tilde:bound}
Let Assumptions \ref{as:initial} to  \ref{as:sig:der} be satisfied. Then,
\begin{align*}
E|\tilde{\sigma} \big(s,X_{\kappa_n(s)}^{i,N,n}, \mu_{\kappa_n(s)}^{X,N,n}\big)|^p \leq K 
\end{align*}
for any $s\in[0,T]$,  $i\in\{1,\cdots,N\}$ and $n, \, N\in\mathbb{N}$ where the  constant $K>0$ does not depend on $n$ and $N$. 
\end{cor}
\begin{cor} \label{cor:one:step:true}
Let Assumptions  \ref{as:initial} to \ref{as:sig:der} be satisfied. Then,
\begin{align*}
E|X_s^{i,N}-X_{\kappa_n(s)}^{i,N}|^p \leq K n^{-\frac{p}{2}} 
\end{align*}
for any $s\in[0,T]$,  $i\in\{1,\cdots,N\}$ and $n, \, N\in\mathbb{N}$ where the  constant $K>0$ does not depend on $n$ and $N$. 
\end{cor}
\begin{cor} \label{cor:one:step}
Let Assumptions  \ref{as:initial} to \ref{as:sig:der} be satisfied. Then,
\begin{align*}
E|X_s^{i,N, n}-X_{\kappa_n(s)}^{i,N, n}|^p \leq K n^{-\frac{p}{2}} 
\end{align*}
for any $s\in[0,T]$,  $i\in\{1,\cdots,N\}$ and $n, \, N\in\mathbb{N}$ where the  constant $K>0$ does not depend on $n$ and $N$. 
\end{cor}
The following  lemma is  very useful in this article. 
\begin{lem}\label{lem:MVT}
Let $f:\mathbb{R}^d\times \mathcal{P}_2(\mathbb{R}^d) \to \mathbb{R}$ be a  continuous function such that its derivative $\partial_x f: \mathbb{R}^d\times \mathcal{P}_2(\mathbb{R}^d)\to\mathbb{R}^d$ and its measure derivative $\partial_\mu f:\mathbb{R}^d\times\mathcal{P}_2(\mathbb{R}^d)\times\mathbb{R}^d \to \mathbb{R}^d$ exists.  
Then, there exists a $\theta \in (0,1)$  such that, 
\begin{align*}
f(y,\mu)&-f(\bar{y},\bar{\mu}) =\big\langle \partial_x f( \bar{y}+ \theta (y-\bar{y}), \mu), y-\bar{y}\big \rangle 
\\
& \qquad+ \tilde{E} \big \langle \partial_\mu f (\bar{y}, L_{\bar{Z}+\theta(Z-\bar{Z})}, \bar{Z}+\theta(Z-\bar{Z})), Z-\bar{Z} \big\rangle
\end{align*}
 for any $y,\bar{y}\in\mathbb{R}^d$, $\mu,\bar{\mu}\in\mathcal{P}_2(\mathbb{R}^d)$,  and  random variables $Z, \bar{Z}$ defined on an atomless, Polish probability space $(\tilde{\Omega}, \tilde{\mathcal{F}}, \tilde{P} )$ such that $L_Z=\mu$ and $L_{\bar{Z}}=\bar{\mu}$.
\end{lem} 
\begin{proof}
Define $\phi:[0,1]\to\mathbb{R}$ by,
\[
\phi(t):=f(\bar{y}+t(y-\bar{y}), \mu)+F(\bar{y}, \bar{Z}+t(Z-\bar{Z}))
\]
 where $F(\bar{y},\cdot)$ is an ``extension" of $f(\bar{y},\cdot)$ on $(\tilde{\Omega}, \tilde{\mathcal{F}}, \tilde{P} )$. For any $t_0\in(0,1)$, we have,
\begin{align*}
\lim_{h\to 0}& \frac{\phi(t_0+h)-\phi(t_0)}{h}=\lim_{h\to 0}\frac{f(\bar{y}+(t_0+h)(y-\bar{y}),\mu)-f(\bar{y}+t_0(y-\bar{y}),\mu)}{h}
\\
&\qquad +\lim_{h\to 0}\frac{F(\bar{y},\bar{Z}+(t_0+h)(Z-\bar{Z}))-F(\bar{y},\bar{Z}+t_0(Z-\bar{Z}))}{h} <\infty
\end{align*}
\textit{i.e.}, $\phi'(t_0)$ exists. Notice that the second term on the right hand side of the above expression is Gateaux derivative of $F$ at $\bar{Z}+t_0 (Z-\bar{Z})$ in the direction of $Z-\bar{Z}$.  Also, $\phi$ is continuous on $[0.1]$. Hence, by mean value theorem, there exists a $\theta\in[0, 1]$ such that $\phi(1)-\phi(0)= \phi'(\theta)$. Thus,
\begin{align*}
f(y&,\mu)-f(\bar{y},\bar{\mu})  =f(y,\mu)+F(\bar{y}, Z)-f(\bar{y},\mu)-F(\bar{y},\bar{Z})=\phi'(\theta)
\end{align*}
completes the proof.  
\end{proof}
As a special case of the above lemma, we obtain the following corollary. 
\begin{cor} \label{cor:MVT}
Let $f:\mathbb{R}^d\times \mathcal{P}_2(\mathbb{R}^d) \to \mathbb{R}$ be a  function such that its derivative $\partial_x f: \mathbb{R}^d\times \mathcal{P}_2(\mathbb{R}^d)\to\mathbb{R}^d$ and its measure derivative $\partial_\mu f:\mathbb{R}^d\times\mathcal{P}_2(\mathbb{R}^d)\times\mathbb{R}^d \to \mathbb{R}^d$ exists.  
Then, there exists a $\theta \in (0,1)$  such that, 
\begin{align*}
f\Big(y, \frac{1}{N}&\sum_{j=1}^N \delta_{z^j} \Big)-f\Big(\bar{y}, \frac{1}{N}\sum_{j=1}^N \delta_{\bar{z}^j} \Big) =\Big\langle \partial_x f\Big( \bar{y}+ \theta( y-\bar{y}), \frac{1}{N}\sum_{j=1}^N \delta_{z^j}\Big), y-\bar{y}\Big\rangle
\\
&+\frac{1}{N}\sum_{j=1}^N \Big\langle \partial_\mu f
\Big(\bar{y}, \frac{1}{N}\sum_{j=1}^N \delta_{\bar{z}^j+\theta (z^j-\bar{z}^j)}, \bar{z}^j+\theta (z^j-\bar{z}^j) \Big),   z^j- \bar{z}^j  \Big\rangle.
\end{align*}
\end{cor}
\begin{proof}
For measures  $\mu=\frac{1}{N}\sum_{j=1}^N \delta_{z^j}$ and $\bar{\mu}=\frac{1}{N}\sum_{j=1}^N \delta_{\bar{z}^j}$,  consider an atomless, Polish probability space $(\tilde{\Omega},\tilde{\mathcal{F}},\tilde{P})$  and define  random variables $Z:\tilde{\Omega} \to \mathbb{R}^d$ and $\bar{Z}:\tilde{\Omega}\to \mathbb{R}^d$ by, 
\begin{align*}
Z:=\sum_{j=1}^N z^j\mathbbm{1}_{\tilde{\Omega}_j}, \,\,\bar{Z}:=\sum_{j=1}^N \bar{z}^j\mathbbm{1}_{\tilde{\Omega}_j}
\end{align*}
where $\{\tilde{\Omega}_1,\ldots,\tilde{\Omega}_N\}$ is a partition of $\tilde{\Omega}$ satisfying $\tilde{P}(\tilde{\Omega}_j)=\frac{1}{N}$ for any $j=1,\ldots,N$.
Clearly,  laws of $Z$ and $\bar{Z}$ satisfy $L_Z=\mu$ and $L_{\bar{Z}}=\bar{\mu}$. Also, 
\[
L_{\bar{Z}+\theta(Z-\bar{Z})}=\frac{1}{N}\sum_{j=1}^N \delta_{\bar{z}^j+\theta (z^j-\bar{z}^j)}.
\]
 By Lemma \ref{lem:MVT},  
\begin{align*}
f\Big(y, & \frac{1}{N}\sum_{j=1}^N \delta_{z^j} \Big)-f\Big(\bar{y}, \frac{1}{N}\sum_{j=1}^N \delta_{\bar{z}^j} \Big)= \Big\langle \partial_x f\Big( \bar{y}+ \theta( y-\bar{y}), \frac{1}{N}\sum_{j=1}^N \delta_{z^j}\Big), y-\bar{y}\Big\rangle
\\
&+\tilde{E}\Big\langle \partial_\mu f
\Big(\bar{y}, \frac{1}{N}\sum_{j=1}^N \delta_{\bar{z}^j+\theta (z^j-\bar{z}^j)}, \sum_{j=1}^N \{\bar{z}_j+\theta (z^j- \bar{z}^j)\} \mathbbm{1}_{\tilde{\Omega}_j}\Big), \sum_{j=1}^N  (z^j- \bar{z}^j) \mathbbm{1}_{\tilde{\Omega}_j} \Big\rangle
\\
&= \Big\langle \partial_x f\Big( \bar{y}+ \theta( y-\bar{y}), \frac{1}{N}\sum_{j=1}^N \delta_{z^j}\Big), y-\bar{y}\Big\rangle
\\
&\qquad+\frac{1}{N}\sum_{j=1}^N \Big\langle \partial_\mu f
\Big(\bar{y}, \frac{1}{N}\sum_{j=1}^N \delta_{\bar{z}^j+\theta (z^j-\bar{z}^j)}, \bar{z}^j+\theta (z^j-\bar{z}^j) \Big),   z^j- \bar{z}^j  \Big\rangle
\end{align*}
which completes the proof. 
\end{proof}
\begin{lem}\label{lem:f:rate}
Let $f:\mathbb{R}^d\times \mathcal{P}_2(\mathbb{R}^d) \to \mathbb{R}$ be a  function such that its derivative $\partial_x f: \mathbb{R}^d\times \mathcal{P}_2(\mathbb{R}^d)\to\mathbb{R}^d$ and measure derivative $\partial_\mu f:\mathbb{R}^d\times\mathcal{P}_2(\mathbb{R}^d)\times\mathbb{R}^d \to \mathbb{R}^d$ satisfy Lipschitz conditions \textit{i.e.}, there exists a constant $L>0$ such that,  
 \begin{align*}
| \partial_x f(x,\mu)-\partial_x f(\bar{x},\bar{\mu})| & \leq L\big\{ |x-\bar{x}| + \mathcal{W}_2(\mu,\bar{\mu}) \big\},
\\
| \partial_\mu f(x,\mu, y)-\partial_\mu f(\bar{x}, \bar{\mu}, \bar{y})| & \leq L\big\{|x-\bar{x}|+ \mathcal{W}_2(\mu,\bar{\mu}) +  |y-\bar{y}| \big\},
 \end{align*}
  for all $x, y,\bar{x}, \bar{y}\in\mathbb{R}^d$ and $\mu,\bar{\mu} \in\mathcal{P}_2(\mathbb{R}^d)$.
Then, 
\begin{align*}
 f\Big(  x^i,\frac{1}{N} & \sum_{j=1}^N\delta_{x^j}\Big)  -f\Big(\bar{x}^i,\frac{1}{N}\sum_{j=1}^N\delta_{\bar{x}^j}\Big)-\Big\langle \partial_x f\Big(\bar{x}^i,\frac{1}{N}\sum_{j=1}^N\delta_{\bar{x}^j}\Big), x^i-\bar{x}^i \Big \rangle  
\\
&\qquad- \frac{1}{N} \sum_{j=1}^N \Big\langle \partial_\mu f\Big(\bar{x}^i,\frac{1}{N}\sum_{j=1}^N \delta_{\bar{x}^j}, \bar{x}^j\Big), x^j-\bar{x}^j \Big \rangle
\\
&\leq  K |x^i-\bar{x}^i|^2 + K \frac{1}{N}\sum_{j=1}^N|x^j-\bar{x}^j|^2
\end{align*}
for every $i \in\{1,\ldots,N\}$ where the constant $K>0$ does not depend on $N\in\mathbb{N}$.
\end{lem}

\begin{proof} By Corollary \ref{cor:MVT}, Cauchy-Schwarz inequality, assumptions on $f$ and Young's inequality,  one obtains,
\begin{align*}
f\Big( & x^i,\frac{1}{N}\sum_{j=1}^N\delta_{x^j}\Big)  -f\Big(\bar{x}^i,\frac{1}{N}\sum_{j=1}^N\delta_{\bar{x}^j}\Big)-\Big\langle \partial_x f\Big(\bar{x}^i,\frac{1}{N}\sum_{j=1}^N\delta_{\bar{x}^j}\Big), x^i-\bar{x}^i \Big \rangle  
\\
&- \frac{1}{N} \sum_{j=1}^N \Big\langle \partial_\mu f\Big(\bar{x}^i,\frac{1}{N}\sum_{j=1}^N \delta_{\bar{x}^j}, \bar{x}^j\Big), x^j-\bar{x}^j \Big \rangle
\\
=&\Big\langle \partial_x  f\Big(\bar{x}^i+\theta(x^i-\bar{x}^i),\frac{1}{N}\sum_{j=1}^N\delta_{x^j}\Big)-\partial_x f\Big(\bar{x}^i,\frac{1}{N}\sum_{j=1}^N\delta_{\bar{x}^j}\Big), x^i-\bar{x}^i \Big \rangle 
\\
&+\frac{1}{N} \sum_{j=1}^N \Big\langle \partial_\mu f\Big(\bar{x}^i,\frac{1}{N}\sum_{j=1}^N \delta_{\bar{x}^j+\theta (x^j-\bar{x}^j)}, \bar{x}^j+\theta(x^j-\bar{x}^j)\Big)
\\
& \qquad- \partial_\mu f\Big(\bar{x}^i,\frac{1}{N}\sum_{j=1}^N \delta_{\bar{x}^j}, \bar{x}^j\Big), x^j-\bar{x}^j \Big \rangle 
\\
\leq & K |x^i-\bar{x}^i|^2 + K \mathcal{W}_2 \Big(\frac{1}{N} \sum_{j=1}^N\delta_{x^j},\frac{1}{N} \sum_{j=1}^N\delta_{\bar{x}^j}\Big)^2
\\
& \qquad+ K \mathcal{W}_2 \Big(\frac{1}{N} \sum_{j=1}^N\delta_{\bar{x}^j+\theta(x^j-\bar{x}^j)},\frac{1}{N} \sum_{j=1}^N\delta_{\bar{x}^j}\Big)^2 +K \frac{1}{N} \sum_{j=1}^N|x^j-\bar{x}^j|^2. 
\end{align*}
The proof is completed by the following estimate on Wasserstein metric,
\begin{align*}
& \mathcal{W}_2 \Big(\frac{1}{N} \sum_{j=1}^N\delta_{x^j},\frac{1}{N} \sum_{j=1}^N\delta_{\bar{x}^j}\Big)^2 
\\
& =\inf_{\pi}\Big\{\int_{\mathbb{R}^d}\int_{\mathbb{R}^d} |x-y|^2 \pi(dx,dy):  \frac{1}{N} \sum_{j=1}^N\delta_{x^j} \mbox{ and } \frac{1}{N} \sum_{j=1}^N\delta_{\bar{x}^j} \mbox{ are marginals of } \pi \Big\}
\\
& \leq \int_{\mathbb{R}^d}\int_{\mathbb{R}^d} |x-y|^2 \frac{1}{N}  \sum_{j=1}^N\delta_{{x}^j}(dx) \delta_{\bar{x}^j}(dy) = \frac{1}{N}\sum_{j=1}^N|x^j-\bar{x}^j|^2.
\end{align*} 
\end{proof}
\begin{lem}\label{lem:f:rate:local}
Let $f:\mathbb{R}^d\times \mathcal{P}_2(\mathbb{R}^d) \to \mathbb{R}$ be a  function such that its derivative $\partial_x f: \mathbb{R}^d\times \mathcal{P}_2(\mathbb{R}^d)\to\mathbb{R}^d$ and measure derivative $\partial_\mu f:\mathbb{R}^d\times\mathcal{P}_2(\mathbb{R}^d)\times\mathbb{R}^d \to \mathbb{R}^d$ satisfy polynomial Lipschitz condition \textit{i.e.}, there exists a constant  $L>0$ such that,  
 \begin{align*}
| \partial_x f(x,\mu)-\partial_x f(\bar{x},\bar{\mu})| & \leq L\big\{ (1+|x|+|\bar{x}|)^{\rho/2}|x-\bar{x}| + \mathcal{W}_2(\mu,\bar{\mu}) \big\},
\\
| \partial_\mu f(x,\mu, y)-\partial_\mu f(\bar{x}, \bar{\mu}, \bar{y})| & \leq L\big\{ (1+|x|+|\bar{x}|)^{\rho/2+1}|x-\bar{x}|+ \mathcal{W}_2(\mu,\bar{\mu}) +  |y-\bar{y}| \big\},
 \end{align*}
  for all $x, y,\bar{x}, \bar{y}\in\mathbb{R}^d$ and $\mu,\bar{\mu} \in\mathcal{P}_2(\mathbb{R}^d)$.
Then, 
\begin{align*}
 f\Big(  x^i,\frac{1}{N} & \sum_{j=1}^N\delta_{x^j}\Big)  -f\Big(\bar{x}^i,\frac{1}{N}\sum_{j=1}^N\delta_{\bar{x}^j}\Big)-\Big\langle \partial_x f\Big(\bar{x}^i,\frac{1}{N}\sum_{j=1}^N\delta_{\bar{x}^j}\Big), x^i-\bar{x}^i \Big \rangle  
\\
&\qquad- \frac{1}{N} \sum_{j=1}^N \Big\langle \partial_\mu f\Big(\bar{x}^i,\frac{1}{N}\sum_{j=1}^N \delta_{\bar{x}^j}, \bar{x}^j\Big), x^j-\bar{x}^j \Big \rangle
\\
&\leq  K (1+|x^i|+|\bar{x}^i|)^{\rho/2} |x^i-\bar{x}^i|^2 + K \frac{1}{N}\sum_{j=1}^N|x^j-\bar{x}^j|^2
\end{align*}
for every $i \in\{1,\ldots,N\}$ where the constant $K>0$ does not depend on $N\in\mathbb{N}$.
\end{lem}
\begin{proof}
The proof follows by adapting the arguments of Lemma \ref{lem:f:rate}. 
\end{proof}
\begin{lem}\label{lem:sigma:rate}
Let Assumptions \ref{as:initial} to \ref{as:sig:der} be satisfied. Then,  for each $i\in\{1,\ldots,N\}$, 
\[
E|\sigma(X_s^{i,N,n},\mu_s^{X,N,n})-\tilde{\sigma}(s,X_{\kappa_n(s)}^{i,N,n},\mu_{\kappa_n(s)}^{X,N,n})|^2 \leq K n^{-2}
\]
for any $s\in[0,T]$ and $n,N\in\mathbb{N}$ where the constant $K>0$ does not depend on $n,N$. 
\end{lem}
\begin{proof} 
From equation \eqref{eq:sigma:tilde}, 
\begin{align} \label{eq:sigma:1}
\sigma^{(u, v)}&(X_s^{i,N,n},\mu_s^{X,N,n})-\tilde{\sigma}^{(u,v)}(s,X_{\kappa_n(s)}^{i,N,n},\mu_{\kappa_n(s)}^{X,N,n}) \notag
\\
 =& \sigma^{(u, v)}(X_s^{i,N,n},\mu_s^{X,N,n})-\sigma^{(u, v)}(X_{\kappa_n(s)}^{i,N,n},\mu_{\kappa_n(s)}^{X,N,n})  \notag
\\
& -  \Big\langle \partial_x \sigma^{(u,v)} (X_{\kappa_n(s)}^{i,N,n},\mu_{\kappa_n(s)}^{X,N,n}), X_{s}^{i,N,n}-X_{\kappa_n(s)}^{i,N,n} \Big\rangle \notag
\\
&  - \frac{1}{N} \sum_{j=1}^N \Big\langle \partial_\mu \sigma^{(u,v)} \big(X_{\kappa_n(s)}^{i,N,n},\mu_{\kappa_n(s)}^{X,N,n}, X_{\kappa_n(s)}^{j,N,n} \big), X_{s}^{j,N,n}-X_{\kappa_n(s)}^{j,N,n} \Big\rangle \notag
\\
&  + \Big\langle \partial_x \sigma^{(u,v)} (X_{\kappa_n(s)}^{i,N,n},\mu_{\kappa_n(s)}^{X,N,n}), X_{s}^{i,N,n}-X_{\kappa_n(s)}^{i,N,n} \Big\rangle \notag
\\
&  + \frac{1}{N} \sum_{j=1}^N \Big\langle \partial_\mu \sigma^{(u,v)} \big(X_{\kappa_n(s)}^{i,N,n},\mu_{\kappa_n(s)}^{X,N,n}, X_{\kappa_n(s)}^{j,N,n} \big), X_{s}^{j,N,n}-X_{\kappa_n(s)}^{j,N,n} \Big\rangle  \notag
\\ 
&- \Lambda_1^{(u, v)} \big(s,X_{\kappa_n(s)}^{i,N,n}, \mu_{\kappa_n(s)}^{X,N,n}\big) - \Lambda_2^{(u, v)} \big(s,X_{\kappa_n(s)}^{i,N,n}, \mu_{\kappa_n(s)}^{X,N,n}\big) 
\end{align} 
almost surely for any $s\in[0,T]$ and $n,N\in \mathbb{N}$. Further, from equation \eqref{eq:scheme}, 
\begin{align*}
&X_s^{i,N, n}-X_{\kappa_n(s)}^{i, N, n}= \int_{\kappa_n(s)}^s b_n\big(X_{\kappa_n(r)}^{i,N, n}, \mu_{\kappa_n(r)}^{X,N,n}\big) dr +\sum_{l=1}^m \int_{\kappa_n(s)}^s \sigma^{(l)}\big(X_{\kappa_n(r)}^{i,N,n}, \mu_{\kappa_n(r)}^{X,N,n}\big) dW_r^{(l),i}
\\
& \quad+\sum_{l=1}^m \int_{\kappa_n(s)}^s \Lambda_1^{(l)}\big(r, X_{\kappa_n(r)}^{i,N,n}, \mu_{\kappa_n(r)}^{X,N,n}\big) dW_r^{(l),i} +\sum_{l=1}^m \int_{\kappa_n(s)}^s \Lambda_2^{(l)}\big(r, X_{\kappa_n(r)}^{i,N,n}, \mu_{\kappa_n(r)}^{X,N,n}\big) dW_r^{(l),i}
\end{align*}
which gives the following expression, 
\begin{align} \label{eq:sigma:2}
 \Big\langle & \partial_x \sigma^{(u,v)} (X_{\kappa_n(s)}^{i,N,n},\mu_{\kappa_n(s)}^{X,N,n}), X_{s}^{i,N,n}-X_{\kappa_n(s)}^{i,N,n} \Big\rangle \notag
 \\
 =& \Big\langle \partial_x \sigma^{(u,v)} (X_{\kappa_n(s)}^{i,N,n},\mu_{\kappa_n(s)}^{X,N,n}), \int_{\kappa_n(s)}^s b_n\big(X_{\kappa_n(r)}^{i,N, n}, \mu_{\kappa_n(r)}^{X,N,n}\big) dr\Big\rangle  + \Lambda_1^{(u, v)} \big(s,X_{\kappa_n(s)}^{i,N,n}, \mu_{\kappa_n(s)}^{X,N,n}\big) \notag
 \\
 & + \Big\langle \partial_x \sigma^{(u,v)} (X_{\kappa_n(s)}^{i,N,n},\mu_{\kappa_n(s)}^{X,N,n}),\sum_{l=1}^m \int_{\kappa_n(s)}^s \Lambda_1^{(l)}\big(r, X_{\kappa_n(r)}^{i,N,n}, \mu_{\kappa_n(r)}^{X,N,n}\big) dW_r^{(l),i}  \Big\rangle \notag
 \\
& + \Big\langle \partial_x \sigma^{(u,v)} (X_{\kappa_n(s)}^{i,N,n},\mu_{\kappa_n(s)}^{X,N,n}), \sum_{l=1}^m \int_{\kappa_n(s)}^s \Lambda_2^{(l)}\big(r, X_{\kappa_n(r)}^{i,N,n}, \mu_{\kappa_n(r)}^{X,N,n}\big) dW_r^{(l),i} \Big\rangle
\end{align}
and similarly, 
\begin{align} \label{eq:sigma:3}
\frac{1}{N} &\sum_{j=1}^N  \Big\langle \partial_\mu \sigma^{(u,v)} \big(X_{\kappa_n(s)}^{i,N,n},\mu_{\kappa_n(s)}^{X,N,n}, X_{\kappa_n(s)}^{j,N,n} \big), X_{s}^{j,N,n}-X_{\kappa_n(s)}^{j,N,n} \Big\rangle \notag
\\
 = & \frac{1}{N} \sum_{j=1}^N \Big\langle \partial_\mu \sigma^{(u,v)} \big(X_{\kappa_n(s)}^{i,N,n},\mu_{\kappa_n(s)}^{X,N,n}, X_{\kappa_n(s)}^{j,N,n} \big), \int_{\kappa_n(s)}^s b_n\big(X_{\kappa_n(r)}^{j,N, n}, \mu_{\kappa_n(r)}^{X,N,n}\big) dr  \Big\rangle \notag 
\\
& +  \Lambda_2^{(u, v)} \big(s,X_{\kappa_n(s)}^{i,N,n}, \mu_{\kappa_n(s)}^{X,N,n}\big) \notag 
\\
& +  \frac{1}{N} \sum_{j=1}^N \Big\langle \partial_\mu \sigma^{(u,v)} \big(X_{\kappa_n(s)}^{i,N,n},\mu_{\kappa_n(s)}^{X,N,n}, X_{\kappa_n(s)}^{j,N,n} \big), \sum_{l=1}^m \int_{\kappa_n(s)}^s \Lambda_1^{(l)}\big(r, X_{\kappa_n(r)}^{j,N,n}, \mu_{\kappa_n(r)}^{X,N,n}\big) dW_r^{(l),j} \Big\rangle \notag
\\
& + \frac{1}{N} \sum_{j=1}^N \Big\langle \partial_\mu \sigma^{(u,v)} \big(X_{\kappa_n(s)}^{i,N,n},\mu_{\kappa_n(s)}^{X,N,n}, X_{\kappa_n(s)}^{j,N,n} \big),\sum_{l=1}^m \int_{\kappa_n(s)}^s \Lambda_2^{(l)}\big(r, X_{\kappa_n(r)}^{j,N,n}, \mu_{\kappa_n(r)}^{X,N,n}\big) dW_r^{(l),j} \Big\rangle
\end{align}
almost surely for any $s\in[0,T]$ and $n,N\in\mathbb{N}$.  On substituting values from equations \eqref{eq:sigma:2} and \eqref{eq:sigma:3} in equation \eqref{eq:sigma:1} and then on using  Lemma \ref{lem:f:rate} and Cauchy-Schwarz inequality, one gets,  
\begin{align} 
& E| \sigma^{(u, v)}(X_s^{i,N,n},\mu_s^{X,N,n})-\tilde{\sigma}^{(u,v)}(s,X_{\kappa_n(s)}^{i,N,n},\mu_{\kappa_n(s)}^{X,N,n}) |^2 \notag
\\
 & \leq K E |X_s^{i,N,n}-X_{\kappa_n(s)}^{i,N,n}|^4  + K \frac{1}{N^2}\sum_{j, k=1}^N E|X_s^{j,N,n}-X_{\kappa_n(s)}^{j,N,n}|^2 |X_s^{k,N,n}-X_{\kappa_n(s)}^{k,N,n}|^2 \notag
 \\
 &  + K E \Big\{|\partial_x \sigma^{(u,v)} (X_{\kappa_n(s)}^{i,N,n},\mu_{\kappa_n(s)}^{X,N,n})| \Big|\int_{\kappa_n(s)}^s b_n\big(X_{\kappa_n(r)}^{i,N, n}, \mu_{\kappa_n(r)}^{X,N,n}\big) dr \Big|\Big\}^2    \notag
 \\
   &  +  K E \Big\{| \partial_x \sigma^{(u,v)} (X_{\kappa_n(s)}^{i,N,n},\mu_{\kappa_n(s)}^{X,N,n})| \sum_{l=1}^m \Big|\int_{\kappa_n(s)}^s \Lambda_1^{(l)}\big(r, X_{\kappa_n(r)}^{i,N,n}, \mu_{\kappa_n(r)}^{X,N,n}\big) dW_r^{(l),i}   \Big|\Big\}^2 \notag
 \\
  &  +K E \Big\{ |\partial_x \sigma^{(u,v)} (X_{\kappa_n(s)}^{i,N,n},\mu_{\kappa_n(s)}^{X,N,n})| \sum_{l=1}^m \Big| \int_{\kappa_n(s)}^s \Lambda_2^{(l)}\big(r, X_{\kappa_n(r)}^{i,N,n}, \mu_{\kappa_n(r)}^{X,N,n}\big) dW_r^{(l),i}  \Big|\Big\}^2  \notag 
\\
  &  + K E\Big\{\frac{1}{N} \sum_{j=1}^N \big|\partial_\mu \sigma^{(u,v)} \big(X_{\kappa_n(s)}^{i,N,n},\mu_{\kappa_n(s)}^{X,N,n}, X_{\kappa_n(s)}^{j,N,n} \big)\big| \Big|\int_{\kappa_n(s)}^s b_n\big(X_{\kappa_n(r)}^{j,N, n}, \mu_{\kappa_n(r)}^{X,N,n}\big) dr  \Big|\Big\}^2 \notag 
\\
  &  + KE\Big\{ \frac{1}{N} \sum_{j=1}^N \big| \partial_\mu \sigma^{(u,v)} \big(X_{\kappa_n(s)}^{i,N,n},\mu_{\kappa_n(s)}^{X,N,n}, X_{\kappa_n(s)}^{j,N,n} \big)\big| \sum_{l=1}^m \Big|\int_{\kappa_n(s)}^s \Lambda_1^{(l)}\big(r, X_{\kappa_n(r)}^{j,N,n}, \mu_{\kappa_n(r)}^{X,N,n}\big) dW_r^{(l),j} \Big|\Big\}^2 \notag
\\
 &  + K E\Big\{  \frac{1}{N} \sum_{j=1}^N \big|\partial_\mu \sigma^{(u,v)} \big(X_{\kappa_n(s)}^{i,N,n},\mu_{\kappa_n(s)}^{X,N,n}, X_{\kappa_n(s)}^{j,N,n} \big)\big| \sum_{l=1}^m \Big|\int_{\kappa_n(s)}^s \Lambda_2^{(l)}\big(r, X_{\kappa_n(r)}^{j,N,n}, \mu_{\kappa_n(r)}^{X,N,n}\big) dW_r^{(l),j} \Big|\Big\}^2 \notag
\end{align} 
for any $s\in[0,T]$ and $n,N\in\mathbb{N}$. Moreover, the application of Corollary \ref{cor:one:step}, Young's inequality, Remarks \ref{rem:b:growth} and \ref{rem:sigma:growth}  yields 
\begin{align} 
 E|  \sigma^{(u, v)}(X_s^{i,N,n},&\mu_s^{X,N,n}) -\tilde{\sigma}^{(u,v)}(s,X_{\kappa_n(s)}^{i,N,n},\mu_{\kappa_n(s)}^{X,N,n}) |^2  \notag
 \\
  \leq &  K n^{-2}  +  K  n^{-2} E  (1+|X_{\kappa_n(s)}^{i,N, n}|)^{\rho+4}    +  K  n^{-2} \frac{1}{N} \sum_{j=1}^N E  (1+|X_{\kappa_n(s)}^{j,N, n}|)^{\rho+4} \notag
\\
& +    K E \sum_{l=1}^m \int_{\kappa_n(s)}^s \big|\Lambda_1^{(l)}\big(r, X_{\kappa_n(r)}^{i,N,n}, \mu_{\kappa_n(r)}^{X,N,n}\big)\big|^2 dr    \notag
 \\
  & +K  E  \sum_{l=1}^m  \int_{\kappa_n(s)}^s | \Lambda_2^{(l)}\big(r, X_{\kappa_n(r)}^{i,N,n}, \mu_{\kappa_n(r)}^{X,N,n}\big)|^2 dr   \notag 
\\
 &  + K  E \frac{1}{N} \sum_{j=1}^N \int_{\kappa_n(s)}^s  \big| \Lambda_1^{(l)}\big(r, X_{\kappa_n(r)}^{j,N,n}, \mu_{\kappa_n(r)}^{X,N,n}\big)  \big|^2 dr  \notag
\\
 & +  K E \frac{1}{N}  \sum_{j=1}^N  \int_{\kappa_n(s)}^s  \big| \Lambda_2^{(l)}\big(r, X_{\kappa_n(r)}^{j,N,n}, \mu_{\kappa_n(r)}^{X,N,n}\big)  \big|^2 dr \notag
\end{align} 
for any $s\in[0,T]$ and $n,N\in\mathbb{N}$. The proof is completed by the application of Lemma \ref{lem:moment:bound}, Corollaries \ref{cor:lambda_1} and \ref{cor:lambda2}.
\end{proof}
\begin{lem} \label{lem:b-b}
Let Assumptions \ref{as:initial} to \ref{as:sig:der} be satisfied. Then,  for each $i\in\{1,\ldots,N\}$, 
\begin{align*}
E|b(X_s^{i,N,n},\mu_s^{X,N,n})-b_n(X_{\kappa_n(s)}^{i,N,n},\mu_{\kappa_n(s)}^{X,N,n})|^2 \leq K n^{-1}
\end{align*}
for any $s\in[0,T]$ and $n,N\in\mathbb{N}$ where the constant $K>0$ does not depend on $n,N$.  
\end{lem}
\begin{proof}
 By Assumption \ref{as:b:lip},  
\begin{align*}
E|b(& X_s^{i,N,n}, \mu_s^{X,N,n})-b_n(X_{\kappa_n(s)}^{i,N,n},\mu_{\kappa_n(s)}^{X,N,n})|^2 \leq K E|b(X_s^{i,N,n},\mu_s^{X,N,n})-b(X_{\kappa_n(s)}^{i,N,n},\mu_{\kappa_n(s)}^{X,N,n})|^2
\\
&+KE|b(X_{\kappa_n(s)}^{i,N,n},\mu_{\kappa_n(s)}^{X,N,n})-b_n(X_{\kappa_n(s)}^{i,N,n},\mu_{\kappa_n(s)}^{X,N,n})|^2
\\
\leq & K E\big(1+|X_s^{i,N,n}|+|X_{\kappa_n(s)}^{i,N,n}| \big)^{\rho+2} |X_s^{i,N,n}-X_{\kappa_n(s)}^{i,N,n}|^2+ K E\mathcal{W}_2\big(\mu_s^{X,N,n},\mu_{\kappa_n(s)}^{X,N,n}\big)^2
\end{align*} 
which on using H\"older's inequality and Corollary \ref{cor:one:step} completes the proof.
\end{proof}
\begin{lem} \label{lem:b-b:x}
Let Assumptions \ref{as:initial} to \ref{as:sig:der} be satisfied. Then, for each $i\in\{1,\ldots,N\}$, 
\begin{align*} 
E \langle X_s^{i,N}  - X_{s}^{i,N,n}, \, & b(X_s^{i,N,n},\mu_s^{X,N,n})-b(X_{\kappa_n(s)}^{i,N,n},\mu_{\kappa_n(s)}^{X,N,n}) \rangle  \leq K n^{-2} 
\\
& + K \sup_{i\in\{1,\cdots,N\}}\sup_{r\in[0,s]} E|X_r^{i,N}-X_r^{i,N,n}|^2 
\end{align*}
for any $s\in[0,T]$ and $n,N\in\mathbb{N}$ where the constant $K>0$ does not depend on $n,N$.  
\end{lem}
\begin{proof}
Note that, 
\begin{align} 
& E \Big\langle  X_s^{i,N}  - X_{s}^{i,N,n}, b\Big(  X_s^{i,N,n},\mu_{s}^{X,N,n}\Big)  -b\Big(X_{\kappa_n(s)}^{i,N,n},\mu_{\kappa_n(s)}^{X,N,n}\Big) \Big\rangle \notag
\\
 & =  E \sum_{k=1}^d  \big(X_s^{(k),i,N}  - X_{s}^{(k), i,N,n}\big)\Big\{  b^{(k)}\big(  X_s^{i,N,n},\mu_{s}^{X,N,n}\big)  -b^{(k)}\big(X_{\kappa_n(s)}^{i,N,n},\mu_{\kappa_n(s)}^{X,N,n}\big) \notag
\\
& \quad- \Big\langle \partial_x b^{(k)}\Big(X_{\kappa_n(s)}^{i,N,n},\mu_{\kappa_n(s)}^{X,N,n}\Big), X_{s}^{i,N,n}- X_{\kappa_n(s)}^{i,N,n} \Big \rangle   \notag
\\
& \quad -  \frac{1}{N} \sum_{j=1}^N \Big\langle \partial_{\mu}b^{(k)}(X_{\kappa_n(s)}^{i,N,n}, \mu_{\kappa_n(s)}^{X,N,n}, X_{\kappa_n(s)}^{j,N,n}), X_{s}^{j,N,n}- X_{\kappa_n(s)}^{j,N,n} \Big \rangle\Big\} \notag
\\
& \quad + E\sum_{k=1}^d \big(X_s^{(k),i,N}  - X_{s}^{(k), i,N,n}\big) \Big\langle \partial_x b^{(k)}\Big(X_{\kappa_n(s)}^{i,N,n},\mu_{\kappa_n(s)}^{X,N,n}\Big), X_{s}^{i,N,n}- X_{\kappa_n(s)}^{i,N,n} \Big \rangle  \notag
\\
&\quad +E\sum_{k=1}^d \big(X_s^{(k),i,N}  - X_{s}^{(k), i,N,n}\big) \frac{1}{N} \sum_{j=1}^N \Big\langle \partial_{\mu}b^{(k)}(X_{\kappa_n(s)}^{i,N,n}, \mu_{\kappa_n(s)}^{X,N,n}, X_{\kappa_n(s)}^{j,N,n}), X_{s}^{j,N,n}- X_{\kappa_n(s)}^{j,N,n} \Big \rangle \notag
\\
&=: T_1+T_2+T_3 \label{eq:T}
\end{align}
for any $s\in[0,T]$ and $n, N\in \mathbb{N}$. By using Young's inequality, $T_1$ is estimated as, 
\begin{align*}
T_1&:= E \sum_{k=1}^d  \big(X_s^{(k),i,N}  - X_{s}^{(k), i,N,n}\big)\Big\{  b^{(k)}\big(  X_s^{i,N,n},\mu_{s}^{X,N,n}\big)  -b^{(k)}\big(X_{\kappa_n(s)}^{i,N,n},\mu_{\kappa_n(s)}^{X,N,n}\big) \notag
\\
& \quad- \Big\langle \partial_x b^{(k)}\Big(X_{\kappa_n(s)}^{i,N,n},\mu_{\kappa_n(s)}^{X,N,n}\Big), X_{s}^{i,N,n}- X_{\kappa_n(s)}^{i,N,n} \Big \rangle   \notag
\\
& \quad -  \frac{1}{N} \sum_{j=1}^N \Big\langle \partial_{\mu}b^{(k)}(X_{\kappa_n(s)}^{i,N,n}, \mu_{\kappa_n(s)}^{X,N,n}, X_{\kappa_n(s)}^{j,N,n}), X_{s}^{j,N,n}- X_{\kappa_n(s)}^{j,N,n} \Big \rangle\Big\} \notag
\\
& \leq K E \sum_{k=1}^d  \big|X_s^{(k),i,N}  - X_{s}^{(k), i,N,n}\big|^2   + K E \sum_{k=1}^d  \Big|  b^{(k)}\big(  X_s^{i,N,n},\mu_{s}^{X,N,n}\big)  -b^{(k)}\big(X_{\kappa_n(s)}^{i,N,n},\mu_{\kappa_n(s)}^{X,N,n}\big) \notag
\\
& \quad- \Big\langle \partial_x b^{(k)}\Big(X_{\kappa_n(s)}^{i,N,n},\mu_{\kappa_n(s)}^{X,N,n}\Big), X_{s}^{i,N,n}- X_{\kappa_n(s)}^{i,N,n} \Big \rangle   \notag
\\
& \quad -  \frac{1}{N} \sum_{j=1}^N \Big\langle \partial_{\mu}b^{(k)}(X_{\kappa_n(s)}^{i,N,n}, \mu_{\kappa_n(s)}^{X,N,n}, X_{\kappa_n(s)}^{j,N,n}), X_{s}^{j,N,n}- X_{\kappa_n(s)}^{j,N,n} \Big \rangle\Big|^2 \notag
\end{align*}
which on using Lemma \ref{lem:f:rate:local}, Corollary \ref{cor:one:step}, Lemma \ref{lem:moment:bound} and H\"older's inequality yields,  
\begin{align}
 T_1  \leq &  K  E  \big|X_s^{i,N}  - X_{s}^{i,N,n}\big|^2   + K E(1+|X_s^{i,N,n}|+|X_{\kappa_n(s)}^{i,N,n}|)^\rho |X_{s}^{i,N,n}- X_{\kappa_n(s)}^{i,N,n}|^4 \notag
 \\
 &+ K \frac{1}{N}\sum_{j=1}^N E|X_{s}^{j,N,n}- X_{\kappa_n(s)}^{j,N,n}|^4 \notag
 \\
 \leq & K \sup_{i\in \{1,\ldots,N\}} \sup_{r\in [0,s]}E  \big|X_r^{i,N}  - X_{r}^{i,N,n}\big|^2 + K n^{-2} \label{eq:T1}
\end{align}
for any $s\in[0,T]$ and $n,N\in\mathbb{N}$.  Further, using equation \eqref{eq:scheme},
\begin{align*}
 T_2  &:=  E\sum_{k=1}^d \big(X_s^{(k),i,N}  - X_{s}^{(k), i,N,n}\big) \Big\langle \partial_x b^{(k)}\Big(X_{\kappa_n(s)}^{i,N,n},\mu_{\kappa_n(s)}^{X,N,n}\Big), X_{s}^{i,N,n}- X_{\kappa_n(s)}^{i,N,n} \Big \rangle  \notag
 \\
 =& E\sum_{k=1}^d \big(X_{s}^{(k),i,N}  - X_{s}^{(k), i,N,n}\big) \Big\langle \partial_x b^{(k)}\Big(X_{\kappa_n(s)}^{i,N,n},\mu_{\kappa_n(s)}^{X,N,n}\Big), \int_{\kappa_n(s)}^s b_n\Big(X_{\kappa_n(r)}^{i,N,n},\mu_{\kappa_n(r)}^{X,N,n}\Big)dr \Big \rangle  \notag
\\
&+E\sum_{k=1}^d \big(X_{s}^{(k),i,N}  - X_{s}^{(k), i,N,n}\big) \Big\langle \partial_x b^{(k)}\Big(X_{\kappa_n(s)}^{i,N,n},\mu_{\kappa_n(s)}^{X,N,n}\Big), \int_{\kappa_n(s)}^s \tilde{\sigma}\Big(r,X_{\kappa_n(r)}^{i,N,n},\mu_{\kappa_n(r)}^{X,N,n}\Big)dW_r^i \Big \rangle  \notag
\\
   = & E\sum_{k=1}^d \big(X_{s}^{(k),i,N}  - X_{s}^{(k), i,N,n}\big) \Big\langle \partial_x b^{(k)}\Big(X_{\kappa_n(s)}^{i,N,n},\mu_{\kappa_n(s)}^{X,N,n}\Big), \int_{\kappa_n(s)}^s b_n\Big(X_{\kappa_n(r)}^{i,N,n},\mu_{\kappa_n(r)}^{X,N,n}\Big)dr \Big \rangle  \notag
\\
&+E\sum_{k=1}^d \big(X_{\kappa_n(s)}^{(k),i,N}  - X_{\kappa_n(s)}^{(k), i,N,n}\big) \Big\langle \partial_x b^{(k)}\Big(X_{\kappa_n(s)}^{i,N,n},\mu_{\kappa_n(s)}^{X,N,n}\Big), \int_{\kappa_n(s)}^s \tilde{\sigma}\Big(r,X_{\kappa_n(r)}^{i,N,n},\mu_{\kappa_n(r)}^{X,N,n}\Big)dW_r^i \Big \rangle  \notag
\\
&+E\sum_{k=1}^d \big(X_{s}^{(k),i,N} - X_{\kappa_n(s)}^{(k),i,N} - X_{s}^{(k), i,N,n}  + X_{\kappa_n(s)}^{(k), i,N,n} \big)
\\
& \qquad \times \Big\langle \partial_x b^{(k)}\Big(X_{\kappa_n(s)}^{i,N,n},\mu_{\kappa_n(s)}^{X,N,n}\Big), \int_{\kappa_n(s)}^s \tilde{\sigma}\Big(r,X_{\kappa_n(r)}^{i,N,n},\mu_{\kappa_n(r)}^{X,N,n}\Big)dW_r^i \Big \rangle  \notag
\end{align*}
for any $s\in[0,T]$ and $n,N\in\mathbb{N}$. Notice that the second term on the right hand side of the above expression  is zero. Thus, from Young's inequality, Remark \ref{rem:bn:growth} and equations \eqref{eq:interactint:particle} and \eqref{eq:scheme}, one obtains 
\begin{align*}
 T_2  \leq &  K E \big|X_{s}^{i,N}  - X_{s}^{i,N,n}\big|^2 +K n^{-2} E  \big(1+ \big|X_{\kappa_n(s)}^{i,N,n}\big|\big)^{2\rho+6}    \notag
\\
& +  E\sum_{k=1}^d \int_{\kappa_n(s)}^s \big\{b^{(k)}(X_r^{i,N},\mu_r^{X,N})-b_n^{(k)}(X_{\kappa_n(r)}^{i,N, n},\mu_{\kappa_n(r)}^{X,N,n})\big\}dr
\\
& \qquad \times \Big\langle \partial_x b^{(k)}\Big(X_{\kappa_n(s)}^{i,N,n},\mu_{\kappa_n(s)}^{X,N,n}\Big), \int_{\kappa_n(s)}^s \tilde{\sigma}\Big(r,X_{\kappa_n(r)}^{i,N,n},\mu_{\kappa_n(r)}^{X,N,n}\Big)dW_r^i \Big \rangle  \notag
\\
&+  E\sum_{k=1}^d \sum_{l=1}^m \int_{\kappa_n(s)}^s \big\{\sigma^{(k,l)}(X_r^{i,N},\mu_r^{X,N})-\tilde{\sigma}^{(k,l)}(r, X_{\kappa_n(r)}^{i,N, n},\mu_{\kappa_n(r)}^{X,N,n})\big\}dW_r^{(l),i}
\\
& \qquad \times \Big\langle \partial_x b^{(k)}\Big(X_{\kappa_n(s)}^{i,N,n},\mu_{\kappa_n(s)}^{X,N,n}\Big), \int_{\kappa_n(s)}^s \tilde{\sigma}\Big(r,X_{\kappa_n(r)}^{i,N,n},\mu_{\kappa_n(r)}^{X,N,n}\Big)dW_r^i \Big \rangle  \notag
\end{align*}
for any $s\in[0,T]$ and $n,N\in\mathbb{N}$.  By using Lemma \ref{lem:moment:bound}, one can write
\begin{align}
T_2  \leq & K E \big|X_{s}^{i,N}  - X_{s}^{i,N,n}\big|^2 +K n^{-2}  \notag
\\
&+ E\sum_{k=1}^d \int_{\kappa_n(s)}^s \big\{b^{(k)}(X_r^{i,N},\mu_r^{X,N})- b^{(k)}(X_r^{i,N,n},\mu_r^{X,N})\big\} dr  \notag
\\
& \qquad \times \Big\langle \partial_x b^{(k)}\Big(X_{\kappa_n(s)}^{i,N,n},\mu_{\kappa_n(s)}^{X,N,n}\Big), \int_{\kappa_n(s)}^s \tilde{\sigma}\Big(r,X_{\kappa_n(r)}^{i,N,n},\mu_{\kappa_n(r)}^{X,N,n}\Big)dW_r^i \Big \rangle  \notag
\\
 + & E\sum_{k=1}^d \int_{\kappa_n(s)}^s \big\{b^{(k)}(X_r^{i,N,n},\mu_r^{X,N})-b^{(k)}(X_r^{i,N,n},\mu_r^{X,N,n})\big\} dr  \notag
\\
& \qquad \times \Big\langle \partial_x b^{(k)}\Big(X_{\kappa_n(s)}^{i,N,n},\mu_{\kappa_n(s)}^{X,N,n}\Big), \int_{\kappa_n(s)}^s \tilde{\sigma}\Big(r,X_{\kappa_n(r)}^{i,N,n},\mu_{\kappa_n(r)}^{X,N,n}\Big)dW_r^i \Big \rangle  \notag
\\
 + & E\sum_{k=1}^d \int_{\kappa_n(s)}^s \big\{b^{(k)}(X_r^{i,N,n},\mu_r^{X,N,n}) -b_n^{(k)}(X_{\kappa_n(r)}^{i,N, n},\mu_{\kappa_n(r)}^{X,N,n})\big\} dr  \notag
\\
& \qquad \times \Big\langle \partial_x b^{(k)}\Big(X_{\kappa_n(s)}^{i,N,n},\mu_{\kappa_n(s)}^{X,N,n}\Big), \int_{\kappa_n(s)}^s \tilde{\sigma}\Big(r,X_{\kappa_n(r)}^{i,N,n},\mu_{\kappa_n(r)}^{X,N,n}\Big)dW_r^i \Big \rangle  \notag
\\
+& E\sum_{k=1}^d \sum_{l=1}^m \int_{\kappa_n(s)}^s \big\{\sigma^{(k,l)}(X_r^{i,N},\mu_r^{X,N})-\sigma^{(k,l)}(X_r^{i,N,n},\mu_r^{X,N}) \big\}dW_r^{(l),i}  \notag
\\
& \qquad \times \Big\langle \partial_x b^{(k)}\Big(X_{\kappa_n(s)}^{i,N,n},\mu_{\kappa_n(s)}^{X,N,n}\Big), \int_{\kappa_n(s)}^s \tilde{\sigma}\Big(r,X_{\kappa_n(r)}^{i,N,n},\mu_{\kappa_n(r)}^{X,N,n}\Big)dW_r^i \Big \rangle  \notag
\\
+& E\sum_{k=1}^d  \sum_{l=1}^m \int_{\kappa_n(s)}^s \big\{\sigma^{(k,l)}(X_r^{i,N,n},\mu_r^{X,N})-\sigma^{(k,l)}(X_r^{i,N,n},\mu_r^{X,N,n}) \big\} dW_r^{(l),i}  \notag
\\
& \qquad \times \Big\langle \partial_x b^{(k)}\Big(X_{\kappa_n(s)}^{i,N,n},\mu_{\kappa_n(s)}^{X,N,n}\Big), \int_{\kappa_n(s)}^s \tilde{\sigma}\Big(r,X_{\kappa_n(r)}^{i,N,n},\mu_{\kappa_n(r)}^{X,N,n}\Big)dW_r^i \Big \rangle  \notag
\\
+&E\sum_{k=1}^d \sum_{l=1}^m \int_{\kappa_n(s)}^s \big\{\sigma^{(k,l)}(X_r^{i,N,n},\mu_r^{X,N, n})-\tilde{\sigma}^{(k,l)}(r, X_{\kappa_n(r)}^{i,N, n},\mu_{\kappa_n(r)}^{X,N,n})\big\} dW_r^{(l),i}  \notag
\\
& \qquad \times \Big\langle \partial_x b^{(k)}\Big(X_{\kappa_n(s)}^{i,N,n},\mu_{\kappa_n(s)}^{X,N,n}\Big), \int_{\kappa_n(s)}^s \tilde{\sigma}\Big(r,X_{\kappa_n(r)}^{i,N,n},\mu_{\kappa_n(r)}^{X,N,n}\Big)dW_r^i \Big \rangle  \notag
\\
=:& \sup_{i\in\{1,\cdots,N\}}\sup_{r\in[0,s]}E \big|X_{s}^{i,N}  - X_{s}^{i,N,n}\big|^2 +K n^{-2} \notag
\\
&  \qquad +T_{21}+T_{22}+T_{23}+T_{24}+T_{25}+T_{26} \label{eq:T2}
\end{align}
for any $s\in[0,T]$ and $n,N\in\mathbb{N}$. 

Using Cauchy-Schwarz inequality, Young's inequality, Assumption~\ref{as:b:lip} and Remark~\ref{rem:b:growth}, $T_{21}$ can be estimated by,
\begin{align*}
T_{21}&:=  E\sum_{k=1}^d \int_{\kappa_n(s)}^s \big\{b^{(k)}(X_r^{i,N},\mu_r^{X,N})- b^{(k)}(X_r^{i,N,n},\mu_r^{X,N})\big\} dr
\\
& \qquad \times \Big\langle \partial_x b^{(k)}\big(X_{\kappa_n(s)}^{i,N,n},\mu_{\kappa_n(s)}^{X,N,n}\big), \int_{\kappa_n(s)}^s \tilde{\sigma}\big(r,X_{\kappa_n(r)}^{i,N,n},\mu_{\kappa_n(r)}^{X,N,n}\big)dW_r^i \Big \rangle  \notag
\\
& \leq  E\sum_{k=1}^d \int_{\kappa_n(s)}^s \big|b^{(k)}(X_r^{i,N},\mu_r^{X,N})- b^{(k)}(X_r^{i,N,n},\mu_r^{X,N})\big| dr
\\
& \qquad \times \big| \partial_x b^{(k)}\big(X_{\kappa_n(s)}^{i,N,n},\mu_{\kappa_n(s)}^{X,N,n}\big)\big| \Big| \int_{\kappa_n(s)}^s \tilde{\sigma}\big(r,X_{\kappa_n(r)}^{i,N,n},\mu_{\kappa_n(r)}^{X,N,n}\big)dW_r^i \Big |  \notag
\\
& \leq K E\int_{\kappa_n(s)}^s  \big(1+|X_r^{i,N}|+| X_r^{i,N,n}|\big)^{\rho/2+1} \big|X_r^{i,N}- X_r^{i,N,n}\big| dr
\\
& \qquad \times  \Big| \int_{\kappa_n(s)}^s \big(1+\big| X_{\kappa_n(s)}^{i,N,n}\big|\big)^{\rho/2+1} \tilde{\sigma}\big(r,X_{\kappa_n(r)}^{i,N,n},\mu_{\kappa_n(r)}^{X,N,n}\big)dW_r^i \Big |  \notag
\\
& \leq K E\int_{\kappa_n(s)}^s  n^{-\frac{1}{2}}\big(1+|X_r^{i,N}|+| X_r^{i,N,n}|\big)^{\rho/2+1} \notag
\\
& \qquad \times \Big| \int_{\kappa_n(s)}^s \big(1+\big| X_{\kappa_n(s)}^{i,N,n}\big|\big)^{\rho/2+1} \tilde{\sigma}\big(r,X_{\kappa_n(r)}^{i,N,n},\mu_{\kappa_n(r)}^{X,N,n}\big)dW_r^i \Big |
 \\
& \qquad \times n^\frac{1}{2}\big|X_r^{i,N}- X_r^{i,N,n}\big| dr
\end{align*}
which on using Young's inequality and H\"older's inequality gives, 
\begin{align*}
T_{21}  \leq &  K E\int_{\kappa_n(s)}^s  n^{-1}\big(1+|X_r^{i,N}|+| X_r^{i,N,n}|\big)^{\rho+2} 
\\
& \qquad \times \Big| \int_{\kappa_n(s)}^s \big(1+\big| X_{\kappa_n(s)}^{i,N,n}\big|\big)^{\rho/2+1} \tilde{\sigma}\big(r,X_{\kappa_n(r)}^{i,N,n},\mu_{\kappa_n(r)}^{X,N,n}\big)dW_r^i \Big |^2 dr
 \\ 
&  + K E\int_{\kappa_n(s)}^s n \big|X_r^{i,N}- X_r^{i,N,n}\big|^2 dr
\\
 \leq &  K \int_{\kappa_n(s)}^s  n^{-1} \big\{E\big(1+|X_r^{i,N}|+| X_r^{i,N,n}|\big)^{2\rho+2}\big\}^\frac{1}{2} 
\\
& \qquad \times \Big\{ E\Big| \int_{\kappa_n(s)}^s \big(1+\big| X_{\kappa_n(s)}^{i,N,n}\big|\big)^{\rho/2+1} \tilde{\sigma}\big(r,X_{\kappa_n(r)}^{i,N,n},\mu_{\kappa_n(r)}^{X,N,n}\big)dW_r^i \Big |^4\Big\}^\frac{1}{2} dr
 \\
& + K \sup_{r\in [0,s]}E\big|X_r^{i,N}- X_r^{i,N,n}\big|^2 
\end{align*}
for any $s\in[0,T]$ and $n,N \in \mathbb{N}$. Further use of Lemma \ref{lem:moment:bound}, Corollary \ref{cor:sigma:tilde:bound} and H\"older's inequality yields,
\begin{align}
T_{21} & \leq K \int_{\kappa_n(s)}^s  n^{-1}  \Big\{ E\Big( \int_{\kappa_n(s)}^s \big(1+\big| X_{\kappa_n(s)}^{i,N,n}\big|\big)^{\rho+2} \big|\tilde{\sigma}\big(r,X_{\kappa_n(r)}^{i,N,n},\mu_{\kappa_n(r)}^{X,N,n}\big) \big|^2 dr \Big)^2\Big\}^\frac{1}{2} dr \notag
 \\
& \qquad + K \sup_{r\in [0,s]}E\big|X_r^{i,N}- X_r^{i,N,n}\big|^2  \notag
\\
& \leq K \int_{\kappa_n(s)}^s  n^{-1}  \Big\{ n^{-1}E \int_{\kappa_n(s)}^s \big(1+\big| X_{\kappa_n(s)}^{i,N,n}\big|\big)^{2\rho+4} \big|\tilde{\sigma}\big(r,X_{\kappa_n(r)}^{i,N,n},\mu_{\kappa_n(r)}^{X,N,n}\big) \big|^4 dr \Big\}^\frac{1}{2} dr \notag
 \\
& \qquad + K \sup_{r\in [0,s]}E\big|X_r^{i,N}- X_r^{i,N,n}\big|^2  \notag
\\
& \leq K \int_{\kappa_n(s)}^s  n^{-1}  \Big\{ n^{-1} \int_{\kappa_n(s)}^s \big\{E\big(1+\big| X_{\kappa_n(s)}^{i,N,n}\big|\big)^{4\rho+8}\big\}^\frac{1}{2} \big\{E\big|\tilde{\sigma}\big(r,X_{\kappa_n(r)}^{i,N,n},\mu_{\kappa_n(r)}^{X,N,n}\big) \big|^8  \big\}^\frac{1}{2} dr \Big\}^\frac{1}{2} dr \notag
 \\
& \qquad + K \sup_{r\in [0,s]}E\big|X_r^{i,N}- X_r^{i,N,n}\big|^2  \notag
\\
& \leq  K \sup_{i\in\{1,\ldots, N\}} \sup_{r\in [0,s]}E\big|X_r^{i,N}- X_r^{i,N,n}\big|^2  + K n^{-3}\label{eq:T21}
\end{align}
for any $s\in[0,T]$ and $n,N\in\mathbb{N}$. 

Using Cauchy-Schwarz inequality, Assumption \ref{as:b:lip} and Remark \ref{rem:b:growth}, $T_{22}$ can be estimated as, 
\begin{align*}
T_{22}& :=  E\sum_{k=1}^d \int_{\kappa_n(s)}^s \big\{b^{(k)}(X_r^{i,N,n},\mu_r^{X,N})-b^{(k)}(X_r^{i,N,n},\mu_r^{X,N,n})\big\} dr
\\
& \qquad \times \Big\langle \partial_x b^{(k)}\Big(X_{\kappa_n(s)}^{i,N,n},\mu_{\kappa_n(s)}^{X,N,n}\Big), \int_{\kappa_n(s)}^s \tilde{\sigma}\big(r,X_{\kappa_n(r)}^{i,N,n},\mu_{\kappa_n(r)}^{X,N,n}\big)dW_r^i \Big \rangle  \notag
\\
& \leq K E\sum_{k=1}^d \int_{\kappa_n(s)}^s \big|b^{(k)}(X_r^{i,N,n},\mu_r^{X,N})-b^{(k)}(X_r^{i,N,n},\mu_r^{X,N,n})\big| dr
\\
& \qquad \times \big| \partial_x b^{(k)}\big(X_{\kappa_n(s)}^{i,N,n},\mu_{\kappa_n(s)}^{X,N,n}\big)\big| \Big| \int_{\kappa_n(s)}^s \tilde{\sigma}\big(r,X_{\kappa_n(r)}^{i,N,n},\mu_{\kappa_n(r)}^{X,N,n}\big)dW_r^i \Big |  \notag
\\
& \leq K E \int_{\kappa_n(s)}^s n^\frac{1}{2}\Big\{\frac{1}{N} \sum_{j=1}^N \big|X_r^{j,N} - X_r^{j,N,n}\big|^2 \Big\}^\frac{1}{2} 
\\
& \qquad \times n^{-\frac{1}{2}} \Big| \int_{\kappa_n(s)}^s \big(1+\big|X_{\kappa_n(s)}^{i,N,n}\big|\big)^{\rho/2+1}  \tilde{\sigma}\big(r,X_{\kappa_n(r)}^{i,N,n},\mu_{\kappa_n(r)}^{X,N,n}\big)dW_r^i \Big |  dr \notag
\end{align*}
which on the application of Young's inequality, Lemma \ref{lem:moment:bound}, H\"older's inequality and Corollary \ref{cor:sigma:tilde:bound} yields,
\begin{align}
T_{22} & \leq K E \int_{\kappa_n(s)}^s n \frac{1}{N} \sum_{j=1}^N \big|X_r^{j,N,n} - X_r^{j,N,n}\big|^2  dr \notag
\\
& \qquad + K n^{-2} E   \Big| \int_{\kappa_n(s)}^s \big(1+\big|X_{\kappa_n(s)}^{i,N,n}\big|\big)^{\rho/2+1}  \tilde{\sigma}\big(r,X_{\kappa_n(r)}^{i,N,n},\mu_{\kappa_n(r)}^{X,N,n}\big)dW_r^i \Big |^2    \notag
\\
& \leq K \sup_{i\in\{1,\cdots,N\}}\sup_{r\in[0,s]}E \big|X_r^{i,N} - X_r^{i,N,n}\big|^2 \notag
\\
& \qquad + K n^{-2} E     \int_{\kappa_n(s)}^s \big(1+\big|X_{\kappa_n(s)}^{i,N,n}\big|\big)^{\rho+2} \big| \tilde{\sigma}\big(r,X_{\kappa_n(r)}^{i,N,n},\mu_{\kappa_n(r)}^{X,N,n}\big)\big|^2 dr    \notag
\\
& \leq K \sup_{i\in\{1,\cdots,N\}}\sup_{r\in[0,s]}E \big|X_r^{i,N} - X_r^{i,N,n}\big|^2   + K n^{-3}    \label{eq:T22}
\end{align}
for any $s\in[0,T]$ and $n,N\in\mathbb{N}$. 

Further, using Cauchy-Schwarz inequality, H\"older's inequality, Lemma \ref{lem:b-b}, Remark~\ref{rem:b:growth}, Corollary \ref{cor:sigma:tilde:bound} and Lemma \ref{lem:moment:bound}, $T_{23}$ can be estimated by, 
\begin{align}
T_{23} & := E\sum_{k=1}^d \int_{\kappa_n(s)}^s \big\{b^{(k)}(X_r^{i,N,n},\mu_r^{X,N,n}) -b_n^{(k)}(X_{\kappa_n(r)}^{i,N, n},\mu_{\kappa_n(r)}^{X,N,n})\big\} dr \notag
\\
& \qquad \times \Big\langle \partial_x b^{(k)}\big(X_{\kappa_n(s)}^{i,N,n},\mu_{\kappa_n(s)}^{X,N,n}\big), \int_{\kappa_n(s)}^s \tilde{\sigma}\big(r,X_{\kappa_n(r)}^{i,N,n},\mu_{\kappa_n(r)}^{X,N,n}\big)dW_r^i \Big \rangle  \notag
\\
& \leq E\int_{\kappa_n(s)}^s \big|b(X_r^{i,N,n},\mu_r^{X,N,n}) -b_n(X_{\kappa_n(r)}^{i,N, n},\mu_{\kappa_n(r)}^{X,N,n})\big| \notag
\\
& \qquad \times \Big| \int_{\kappa_n(s)}^s \big(1+ \big| X_{\kappa_n(s)}^{i,N,n}\big|\big)^{\rho/2+1}  \tilde{\sigma}\big(r,X_{\kappa_n(r)}^{i,N,n},\mu_{\kappa_n(r)}^{X,N,n}\big)dW_r^i \Big | dr \notag
\\
& \leq \int_{\kappa_n(s)}^s \big\{E\big|b(X_r^{i,N,n},\mu_r^{X,N,n}) -b_n(X_{\kappa_n(r)}^{i,N, n},\mu_{\kappa_n(r)}^{X,N,n})\big|^2\big\}^\frac{1}{2} \notag
\\
& \qquad \times \Big\{ E\Big| \int_{\kappa_n(s)}^s \big(1+ \big| X_{\kappa_n(s)}^{i,N,n}\big|\big)^{\rho/2+1}  \tilde{\sigma}\big(r,X_{\kappa_n(r)}^{i,N,n},\mu_{\kappa_n(r)}^{X,N,n}\big)dW_r^i \Big |^2 \Big\}^\frac{1}{2} dr \notag
\\
& \leq K n^{-3/2}  \Big\{ E \int_{\kappa_n(s)}^s \big(1+ \big| X_{\kappa_n(s)}^{i,N,n}\big|\big)^{\rho+2} \big|  \tilde{\sigma}\big(r,X_{\kappa_n(r)}^{i,N,n},\mu_{\kappa_n(r)}^{X,N,n}\big) \big|^2 dr \Big\}^\frac{1}{2} dr  \leq K n^{-2} \label{eq:T23}
\end{align}
for any $s\in[0,T]$ and $n,N\in\mathbb{N}$. 

For estimating $T_{24}$, one applies Cauchy-Schwarz inequality, Young's inequality and Remark \ref{rem:b:growth} to get, 
\begin{align*}
T_{24} & :=E\sum_{k=1}^d \sum_{l=1}^m \int_{\kappa_n(s)}^s  \big\{\sigma^{(k,l)}(X_r^{i,N},\mu_r^{X,N})-\sigma^{(k,l)}(X_r^{i,N,n},\mu_r^{X,N}) \big\}dW_r^{(l),i}
\\
& \qquad \times \Big\langle \partial_x b^{(k)}\big(X_{\kappa_n(s)}^{i,N,n},\mu_{\kappa_n(s)}^{X,N,n}\big), \int_{\kappa_n(s)}^s \tilde{\sigma}\big(r,X_{\kappa_n(r)}^{i,N,n},\mu_{\kappa_n(r)}^{X,N,n}\big)dW_r^i \Big \rangle  \notag
\\
& = E\sum_{k=1}^d \sum_{l=1}^m \int_{\kappa_n(s)}^s  \big\{\sigma^{(k,l)}(X_r^{i,N},\mu_r^{X,N})-\sigma^{(k,l)}(X_r^{i,N,n},\mu_r^{X,N}) \big\}dW_r^{(l),i}
\\
& \qquad \times  \sum_{l_1=1}^m  \int_{\kappa_n(s)}^s \Big\langle \partial_x b^{(k)}\big(X_{\kappa_n(s)}^{i,N,n},\mu_{\kappa_n(s)}^{X,N,n}\big), \tilde{\sigma}^{(l_1)}\big(r,X_{\kappa_n(r)}^{i,N,n},\mu_{\kappa_n(r)}^{X,N,n}\big) \Big \rangle  dW_r^{(l_1),i} \notag
\\
& =  E\sum_{k=1}^d \sum_{l=1}^d\int_{\kappa_n(s)}^s  \big\{\sigma^{(k,l)}(X_r^{i,N},\mu_r^{X,N})-\sigma^{(k,l)}(X_r^{i,N,n},\mu_r^{X,N}) \big\}
\\
& \qquad \times    \Big\langle \partial_x b^{(k)}\big(X_{\kappa_n(s)}^{i,N,n},\mu_{\kappa_n(s)}^{X,N,n}\big), \tilde{\sigma}^{(l)}\big(r,X_{\kappa_n(r)}^{i,N,n},\mu_{\kappa_n(r)}^{X,N,n}\big) \Big \rangle  dr \notag
\\
& \leq K  E\sum_{k=1}^d \sum_{l=1}^d\int_{\kappa_n(s)}^s n^{1/2} \big|\sigma^{(k,l)}(X_r^{i,N},\mu_r^{X,N})-\sigma^{(k,l)}(X_r^{i,N,n},\mu_r^{X,N}) \big|
\\
& \qquad \times   n^{-1/2} \big| \partial_x b^{(k)}\big(X_{\kappa_n(s)}^{i,N,n},\mu_{\kappa_n(s)}^{X,N,n}\big)\big| \big| \tilde{\sigma}^{(l)}\big(r,X_{\kappa_n(r)}^{i,N,n},\mu_{\kappa_n(r)}^{X,N,n}\big) \big |  dr \notag
\\
 & \leq  K E \int_{\kappa_n(s)}^s n \big| X_r^{i,N}-X_r^{i,N,n} \big|^2 dr \notag
\\
& \qquad + K E\sum_{l=1}^d \int_{\kappa_n(s)}^s n^{-1}   \big(1+ \big|X_{\kappa_n(r)}^{i,N,n} \big| \big)^{\rho+2} \big|\tilde{\sigma}^{(l)}\big(r,X_{\kappa_n(r)}^{i,N,n},\mu_{\kappa_n(r)}^{X,N,n}\big) \big |^2  dr \notag
\end{align*}
which on using  H\"older's inequality, Corollary \ref{cor:sigma:tilde:bound} and Lemma \ref{lem:moment:bound} yields,
\begin{align} \label{eq:T24}
T_{24} & \leq K \sup_{i\in\{1,\cdots,N\}}\sup_{r\in[0,s]}E  \big| X_r^{i,N}-X_r^{i,N,n} \big|^2  + K n^{-2}   
\end{align}
for any $s\in[0,T]$ and $n,N\in\mathbb{N}$. 

Again notice that, 
\begin{align*}
T_{25}& := E\sum_{k=1}^d  \sum_{l=1}^m\int_{\kappa_n(s)}^s \big\{\sigma^{(k,l)}(X_r^{i,N,n},\mu_r^{X,N})-\sigma^{(k,l)}(X_r^{i,N,n},\mu_r^{X,N,n}) \big\} dW_r^{(l),i}
\\
& \qquad \times \Big\langle \partial_x b^{(k)}\big(X_{\kappa_n(s)}^{i,N,n},\mu_{\kappa_n(s)}^{X,N,n}\big), \int_{\kappa_n(s)}^s \tilde{\sigma}\big(r,X_{\kappa_n(r)}^{i,N,n},\mu_{\kappa_n(r)}^{X,N,n}\big)dW_r^i \Big \rangle  \notag
\\
& = E\sum_{k=1}^d  \sum_{l=1}^m \int_{\kappa_n(s)}^s \big\{\sigma^{(k,l)}(X_r^{i,N,n},\mu_r^{X,N})-\sigma^{(k,l)}(X_r^{i,N,n},\mu_r^{X,N,n}) \big\} dW_r^{(l),i}
\\
& \qquad \times \sum_{l_1=1}^m \int_{\kappa_n(s)}^s \Big\langle \partial_x b^{(k)}\Big(X_{\kappa_n(s)}^{i,N,n},\mu_{\kappa_n(s)}^{X,N,n}\Big),  \tilde{\sigma}^{(l_1)}\big(r,X_{\kappa_n(r)}^{i,N,n},\mu_{\kappa_n(r)}^{X,N,n}\big)\Big \rangle  dW_r^{(l_1),i}  \notag
\\
& = E\sum_{k=1}^d  \sum_{l=1}^m \int_{\kappa_n(s)}^s \big\{\sigma^{(k,l)}(X_r^{i,N,n},\mu_r^{X,N})-\sigma^{(k,l)}(X_r^{i,N,n},\mu_r^{X,N,n}) \big\} 
\\
& \qquad \times \Big\langle \partial_x b^{(k)}\big(X_{\kappa_n(s)}^{i,N,n},\mu_{\kappa_n(s)}^{X,N,n}\big),  \tilde{\sigma}^{(l)}\big(r,X_{\kappa_n(r)}^{i,N,n},\mu_{\kappa_n(r)}^{X,N,n}\big)\Big \rangle  dr \notag
\end{align*}
which by applying Cauchy-Schwarz inequality, Assumption \ref{as:sig:lip}, Remark \ref{rem:b:growth}, Young's inequality, H\"older's inequality, Lemma \ref{lem:moment:bound} and Corollary \ref{cor:sigma:tilde:bound} gives,  
\begin{align}
T_{25} & \leq K E \sum_{l=1}^m \int_{\kappa_n(s)}^s \mathcal{W}_2\big(\mu_r^{X,N}, \mu_r^{X,N,n}\big)  \big(1+\big|X_{\kappa_n(s)}^{i,N,n}\big|\big)^{\rho/2+1} \big|  \tilde{\sigma}^{(l)}\big(r,X_{\kappa_n(r)}^{i,N,n},\mu_{\kappa_n(r)}^{X,N,n}\big)\big |  dr \notag
\\
& \leq K E \sum_{l=1}^m \int_{\kappa_n(s)}^s n^\frac{1}{2} \Big\{\frac{1}{N} \sum_{j=1}^N \big|X_{r}^{i,N}-X_{r}^{i,N,n}\big|^2\Big\}^\frac{1}{2} \notag
\\
& \qquad \times  n^{-\frac{1}{2} }\big(1+\big|X_{\kappa_n(s)}^{i,N,n}\big|\big)^{\rho/2+1} \big|  \tilde{\sigma}^{(l)}\big(r,X_{\kappa_n(r)}^{i,N,n},\mu_{\kappa_n(r)}^{X,N,n}\big)\big |  dr \notag
\\
& \leq K E \sum_{l=1}^m \int_{\kappa_n(s)}^s n \frac{1}{N} \sum_{j=1}^N \big|X_{r}^{i,N}-X_{r}^{i,N,n}\big|^2 dr \notag
\\
& \qquad +  K E \sum_{l=1}^m \int_{\kappa_n(s)}^s n^{-1} \big(1+\big|X_{\kappa_n(s)}^{i,N,n}\big|\big)^{\rho+2} \big|  \tilde{\sigma}^{(l)}\big(r,X_{\kappa_n(r)}^{i,N,n},\mu_{\kappa_n(r)}^{X,N,n}\big)\big |^2  dr \notag
\\
& \leq K \sup_{i\in\{1,\cdots,N\}}\sup_{r\in[0,s]}E  \big|X_{r}^{i,N}-X_{r}^{i,N,n}\big|^2  +  K  n^{-2} \label{eq:T25}
\end{align}
for any $s\in[0,T]$ and $n,N\in\mathbb{N}$. 

Proceeding as before, by Cauchy-Schwarz inequality and Remark~\ref{rem:b:growth},
\begin{align*}
T_{26}& := E\sum_{k=1}^d \sum_{l=1}^m \int_{\kappa_n(s)}^s \big\{\sigma^{(k,l)}(X_r^{i,N,n},\mu_r^{X,N, n})-\tilde{\sigma}^{(k,l)}(r, X_{\kappa_n(r)}^{i,N, n},\mu_{\kappa_n(r)}^{X,N,n})\big\} dW_r^{(l),i}
\\
& \qquad \times \Big\langle \partial_x b^{(k)}\big(X_{\kappa_n(s)}^{i,N,n},\mu_{\kappa_n(s)}^{X,N,n}\big), \int_{\kappa_n(s)}^s \tilde{\sigma}\big(r,X_{\kappa_n(r)}^{i,N,n},\mu_{\kappa_n(r)}^{X,N,n}\big)dW_r^i \Big \rangle  \notag
\\
& = E\sum_{k=1}^d \sum_{l=1}^m \int_{\kappa_n(s)}^s \big\{\sigma^{(k,l)}(X_r^{i,N,n},\mu_r^{X,N, n})-\tilde{\sigma}^{(k,l)}(r, X_{\kappa_n(r)}^{i,N, n},\mu_{\kappa_n(r)}^{X,N,n})\big\} dW_r^{(l),i}
\\
& \qquad \times \sum_{l_1=1}^m   \int_{\kappa_n(s)}^s \Big\langle \partial_x b^{(k)}\big(X_{\kappa_n(s)}^{i,N,n},\mu_{\kappa_n(s)}^{X,N,n}\big), \tilde{\sigma}^{(l_1)}\big(r,X_{\kappa_n(r)}^{i,N,n},\mu_{\kappa_n(r)}^{X,N,n}\big) \Big \rangle  dW_r^{(l_1),i} \notag
\\
& =E\sum_{k=1}^d \sum_{l=1}^m \int_{\kappa_n(s)}^s \big\{\sigma^{(k,l)}(X_r^{i,N,n},\mu_r^{X,N, n})-\tilde{\sigma}^{(k,l)}(r, X_{\kappa_n(r)}^{i,N, n},\mu_{\kappa_n(r)}^{X,N,n})\big\} 
\\
& \qquad \times  \Big\langle \partial_x b^{(k)}\big(X_{\kappa_n(s)}^{i,N,n},\mu_{\kappa_n(s)}^{X,N,n}\big), \tilde{\sigma}^{(l)}\big(r,X_{\kappa_n(r)}^{i,N,n},\mu_{\kappa_n(r)}^{X,N,n}\big) \Big \rangle  dr \notag
\\
& \leq KE\sum_{k=1}^d \sum_{l=1}^m \int_{\kappa_n(s)}^s   \big|\sigma^{(k,l)}(X_r^{i,N,n},\mu_r^{X,N, n})-\tilde{\sigma}^{(k,l)}(r, X_{\kappa_n(r)}^{i,N, n},\mu_{\kappa_n(r)}^{X,N,n})\big| 
\\
& \qquad \times  \big(1+ |X_{\kappa_n(s)}^{i,N,n}\big|\big)^\frac{\rho}{2} \big|\tilde{\sigma}^{(l)}\big(r,X_{\kappa_n(r)}^{i,N,n},\mu_{\kappa_n(r)}^{X,N,n}\big) \big |  dr \notag
\end{align*}
which on using H\"older's inequality, Lemma \ref{lem:sigma:rate}, Lemma \ref{lem:moment:bound} and Corollary \ref{cor:sigma:tilde:bound} gives,
\begin{align}
T_{26} & \leq KE\sum_{k=1}^d \sum_{l=1}^m \int_{\kappa_n(s)}^s  \big\{ E\big|\sigma^{(k,l)}(X_r^{i,N,n},\mu_r^{X,N, n})-\tilde{\sigma}^{(k,l)}(r, X_{\kappa_n(r)}^{i,N, n},\mu_{\kappa_n(r)}^{X,N,n})\big|^2 \big\}^\frac{1}{2} dr  \notag
\\
& \qquad \times \big\{E\big(1+ |X_{\kappa_n(s)}^{i,N,n}\big|\big)^{2\rho} \big\}^\frac{1}{4} \big\{E\big|\tilde{\sigma}^{(l)}\big(r,X_{\kappa_n(r)}^{i,N,n},\mu_{\kappa_n(r)}^{X,N,n}\big) \big |^4\big\}^\frac{1}{4}  dr \leq K n^{-2} \label{eq:T26}
\end{align}
for any $s\in[0,T]$ and $n,N\in\mathbb{N}$. Substituting values from \eqref{eq:T21} to \eqref{eq:T26} in \eqref{eq:T2} gives, 
\begin{align} \label{eq:T2:final}
T_2 \leq \sup_{i\in\{1,\cdots,N\}}\sup_{r\in[0,s]} E|X_r^{i,N}-X_r^{i,N,n}|^2 + K n^{-2}
\end{align}
for any $s\in[0,T]$ and $n,N\in\mathbb{N}$. 

Finally, we proceed with estimation of $T_3$ of equation \eqref{eq:T}.  By Young's inequality and Cauchy-Schwarz inequality,
\begin{align*}
T_3  :=& E\sum_{k=1}^d \big(X_s^{(k),i,N}  - X_{s}^{(k), i,N,n}\big) \frac{1}{N} \sum_{j=1}^N \Big\langle \partial_\mu b^{(k)}\big(X_{\kappa_n(s)}^{i,N,n},\mu_{\kappa_n(s)}^{X,N,n}, X_{\kappa_n(s)}^{j,N,n}\big), X_{s}^{j,N,n}- X_{\kappa_n(s)}^{j,N,n} \Big \rangle \notag
\\
 =&  E\sum_{k=1}^d \big(X_s^{(k),i,N}  - X_{s}^{(k), i,N,n}\big)  
 \\
 & \qquad \times \frac{1}{N} \sum_{j=1}^N \Big\langle \partial_\mu b^{(k)}\big(X_{\kappa_n(s)}^{i,N,n},\mu_{\kappa_n(s)}^{X,N,n}, X_{\kappa_n(s)}^{j,N,n}\big), \int_{\kappa_n(s)}^{s} b_n(X_{\kappa_n(r)}^{i,N,n}, \mu_{\kappa_n(r)}^{X,N,n}) dr  \Big \rangle \notag
\\
&+ E\sum_{k=1}^d \big(X_s^{(k),i,N}  - X_{s}^{(k), i,N,n}\big) 
\\
& \qquad \times  \frac{1}{N} \sum_{j=1}^N \Big\langle \partial_\mu b^{(k)}\big(X_{\kappa_n(s)}^{i,N,n},\mu_{\kappa_n(s)}^{X,N,n}, X_{\kappa_n(s)}^{j,N,n}\big), \int_{\kappa_n(s)}^{s} \tilde{\sigma}(r, X_{\kappa_n(r)}^{i,N,n}, \mu_{\kappa_n(r)}^{X,N,n})dW_r^i  \Big \rangle \notag
\\
 \leq&  K E\big|X_s^{i,N}  - X_{s}^{i,N,n}\big|^2 
\\
& + K \sum_{k=1}^d  E \frac{1}{N^2} N \sum_{j=1}^N \big|  \partial_\mu b^{(k)}\big(X_{\kappa_n(s)}^{i,N,n},\mu_{\kappa_n(s)}^{X,N,n}, X_{\kappa_n(s)}^{j,N,n}\big)\big|^2 \Big| \int_{\kappa_n(s)}^{s} b_n(X_{\kappa_n(r)}^{i,N,n}, \mu_{\kappa_n(r)}^{X,N,n}) dr  \Big|^2  \notag
\\
&+  E\sum_{k=1}^d \big(X_{\kappa_n(s)}^{(k),i,N}  - X_{\kappa_n(s)}^{(k), i,N,n}\big)
\\
& \qquad \times  \frac{1}{N} \sum_{j=1}^N \Big\langle \partial_\mu b^{(k)}\big(X_{\kappa_n(s)}^{i,N,n},\mu_{\kappa_n(s)}^{X,N,n}, X_{\kappa_n(s)}^{j,N,n}\big), \int_{\kappa_n(s)}^{s} \tilde{\sigma}(r, X_{\kappa_n(r)}^{i,N,n}, \mu_{\kappa_n(r)}^{X,N,n})dW_r^i  \Big \rangle \notag
\\
&+  E\sum_{k=1}^d \big(X_s^{(k),i,N}  - X_{s}^{(k), i,N,n}-X_{\kappa_n(s)}^{(k),i,N}  +X_{\kappa_n(s)}^{(k), i,N,n}\big) 
\\
& \qquad \times \frac{1}{N} \sum_{j=1}^N \Big\langle \partial_\mu b^{(k)}\big(X_{\kappa_n(s)}^{i,N,n},\mu_{\kappa_n(s)}^{X,N,n}, X_{\kappa_n(s)}^{j,N,n}\big), \int_{\kappa_n(s)}^{s} \tilde{\sigma}(r, X_{\kappa_n(r)}^{i,N,n}, \mu_{\kappa_n(r)}^{X,N,n})dW_r^i  \Big \rangle \notag
\end{align*}
for any $s\in[0,T]$ and $n,N\in\mathbb{N}$. 
Notice that third term on the right hand side of the above expression is zero. Further, by Remark \ref{rem:b:growth}, Remark \ref{rem:bn:growth} and Lemma \ref{lem:moment:bound},  one obtains
\begin{align*}
T_3 & \leq K E\big|X_s^{i,N}  - X_{s}^{i,N,n}\big|^2   + K n^{-2}
\\
& +  E\sum_{k=1}^d \int_{\kappa_n(s)}^s \big\{b^{(k)}\big(X_r^{i,N}, \mu_r^{X,N}\big)-b_n^{(k)}\big(X_{\kappa_n(r)}^{i,N, n}, \mu_{\kappa_n(r)}^{X,N,n}\big)\big\} dr
\\
& \qquad \times \frac{1}{N} \sum_{j=1}^N \Big\langle \partial_\mu b^{(k)}\big(X_{\kappa_n(s)}^{i,N,n},\mu_{\kappa_n(s)}^{X,N,n}, X_{\kappa_n(s)}^{j,N,n}\big), \int_{\kappa_n(s)}^{s} \tilde{\sigma}(r, X_{\kappa_n(r)}^{i,N,n}, \mu_{\kappa_n(r)}^{X,N,n})dW_r^i  \Big \rangle \notag
\\
& +  E\sum_{k=1}^d \sum_{l=1}^m \int_{\kappa_n(s)}^s \big\{\sigma^{(k,l)}\big(X_r^{i,N}, \mu_r^{X,N}\big)-\tilde{\sigma}^{(k,l)}\big(r, X_{\kappa_n(r)}^{i,N, n}, \mu_{\kappa_n(r)}^{X,N,n}\big)\big\} dW_r^{(l),i}
\\
& \qquad \times \frac{1}{N} \sum_{j=1}^N \Big\langle \partial_\mu b^{(k)}\big(X_{\kappa_n(s)}^{i,N,n},\mu_{\kappa_n(s)}^{X,N,n}, X_{\kappa_n(s)}^{j,N,n}\big), \int_{\kappa_n(s)}^{s} \tilde{\sigma}(r, X_{\kappa_n(r)}^{i,N,n}, \mu_{\kappa_n(r)}^{X,N,n})dW_r^i  \Big \rangle \notag
\end{align*}
for any $s\in[0,T]$ and $n,N\in\mathbb{N}$. The above expression can further be written as, 
\begin{align}
T_3 & \leq K E\big|X_s^{i,N}  - X_{s}^{i,N,n}\big|^2  + K n^{-2} \notag
\\
& +  E\sum_{k=1}^d \int_{\kappa_n(s)}^s \big\{b^{(k)}\big(X_r^{i,N}, \mu_r^{X,N}\big)-b^{(k)}\big(X_{r}^{i,N, n}, \mu_{r}^{X,N}\big)\big\} dr \notag
\\
& \qquad \times \frac{1}{N} \sum_{j=1}^N \Big\langle \partial_\mu b^{(k)}\big(X_{\kappa_n(s)}^{i,N,n},\mu_{\kappa_n(s)}^{X,N,n}, X_{\kappa_n(s)}^{j,N,n}\big), \int_{\kappa_n(s)}^{s} \tilde{\sigma}(r, X_{\kappa_n(r)}^{i,N,n}, \mu_{\kappa_n(r)}^{X,N,n})dW_r^i  \Big \rangle \notag
\\
& +  E\sum_{k=1}^d \int_{\kappa_n(s)}^s \big\{b^{(k)}\big(X_r^{i,N,n}, \mu_r^{X,N}\big)-b^{(k)}\big(X_{r}^{i,N, n}, \mu_{r}^{X,N,n}\big)\big\} dr \notag
\\
& \qquad \times \frac{1}{N} \sum_{j=1}^N \Big\langle \partial_\mu b^{(k)}\big(X_{\kappa_n(s)}^{i,N,n},\mu_{\kappa_n(s)}^{X,N,n}, X_{\kappa_n(s)}^{j,N,n}\big), \int_{\kappa_n(s)}^{s} \tilde{\sigma}(r, X_{\kappa_n(r)}^{i,N,n}, \mu_{\kappa_n(r)}^{X,N,n})dW_r^i  \Big \rangle \notag
\\
& + E\sum_{k=1}^d \int_{\kappa_n(s)}^s \big\{b^{(k)}\big(X_r^{i,N,n}, \mu_r^{X,N,n}\big)-b_n^{(k)}\big(X_{\kappa_n(r)}^{i,N, n}, \mu_{\kappa_n(r)}^{X,N,n}\big)\big\} dr \notag
\\
& \qquad \times \frac{1}{N} \sum_{j=1}^N \Big\langle \partial_\mu b^{(k)}\big(X_{\kappa_n(s)}^{i,N,n},\mu_{\kappa_n(s)}^{X,N,n}, X_{\kappa_n(s)}^{j,N,n}\big), \int_{\kappa_n(s)}^{s} \tilde{\sigma}(r, X_{\kappa_n(r)}^{i,N,n}, \mu_{\kappa_n(r)}^{X,N,n})dW_r^i  \Big \rangle \notag
\\
& +  E\sum_{k=1}^d \sum_{l=1}^m \int_{\kappa_n(s)}^s \big\{\sigma^{(k,l)}\big(X_r^{i,N}, \mu_r^{X,N}\big)-\sigma^{(k,l)}\big(X_{r}^{i,N, n}, \mu_{r}^{X,N}\big)\big\} dW_r^{(l),i} \notag
\\
& \qquad \times \frac{1}{N} \sum_{j=1}^N \Big\langle \partial_\mu b^{(k)}\big(X_{\kappa_n(s)}^{i,N,n},\mu_{\kappa_n(s)}^{X,N,n}, X_{\kappa_n(s)}^{j,N,n}\big), \int_{\kappa_n(s)}^{s} \tilde{\sigma}(r, X_{\kappa_n(r)}^{i,N,n}, \mu_{\kappa_n(r)}^{X,N,n})dW_r^i  \Big \rangle \notag
\\
& +  E\sum_{k=1}^d \sum_{l=1}^m \int_{\kappa_n(s)}^s \big\{\sigma^{(k,l)}\big(X_r^{i,N,n}, \mu_r^{X,N}\big)-\sigma^{(k,l)}\big(X_{r}^{i,N, n}, \mu_{r}^{X,N,n}\big)\big\} dW_r^{(l),i} \notag
\\
& \qquad \times \frac{1}{N} \sum_{j=1}^N \Big\langle \partial_\mu b^{(k)}\big(X_{\kappa_n(s)}^{i,N,n},\mu_{\kappa_n(s)}^{X,N,n}, X_{\kappa_n(s)}^{j,N,n}\big), \int_{\kappa_n(s)}^{s} \tilde{\sigma}(r, X_{\kappa_n(r)}^{i,N,n}, \mu_{\kappa_n(r)}^{X,N,n})dW_r^i  \Big \rangle \notag
\\
& +  E\sum_{k=1}^d \sum_{l=1}^m \int_{\kappa_n(s)}^s \big\{\sigma^{(k,l)}\big(X_r^{i,N,n}, \mu_r^{X,N,n}\big)-\tilde{\sigma}^{(k,l)}\big(r, X_{\kappa_n(r)}^{i,N, n}, \mu_{\kappa_n(r)}^{X,N,n}\big)\big\} dW_r^{(l),i} \notag
\\
& \qquad \times \frac{1}{N} \sum_{j=1}^N \Big\langle \partial_\mu b^{(k)}\big(X_{\kappa_n(s)}^{i,N,n},\mu_{\kappa_n(s)}^{X,N,n}, X_{\kappa_n(s)}^{j,N,n}\big), \int_{\kappa_n(s)}^{s} \tilde{\sigma}(r, X_{\kappa_n(r)}^{i,N,n}, \mu_{\kappa_n(r)}^{X,N,n})dW_r^i  \Big \rangle \notag
\\
& =: K E\big|X_s^{i,N}  - X_{s}^{i,N,n}\big|^2  + K n^{-2} +T_{31}+T_{32}+T_{33}+T_{34}+T_{35}+T_{36} \label{eq:T3}
\end{align}
for any $s\in[0,T]$ and $n,N\in\mathbb{N}$. 

By Cauchy-Schwarz inequality, Assumption \ref{as:b:lip}, Remark \ref{rem:b:growth}
and Young's inequality, 
\begin{align*}
T_{31}& :=  E\sum_{k=1}^d \int_{\kappa_n(s)}^s \big\{b^{(k)}\big(X_r^{i,N}, \mu_r^{X,N}\big)-b^{(k)}\big(X_{r}^{i,N, n}, \mu_{r}^{X,N}\big)\big\} dr
\\
& \qquad \times \frac{1}{N} \sum_{j=1}^N \Big\langle \partial_\mu b^{(k)}\big(X_{\kappa_n(s)}^{i,N,n},\mu_{\kappa_n(s)}^{X,N,n}, X_{\kappa_n(s)}^{j,N,n}\big), \int_{\kappa_n(s)}^{s} \tilde{\sigma}(r, X_{\kappa_n(r)}^{i,N,n}, \mu_{\kappa_n(r)}^{X,N,n})dW_r^i  \Big \rangle \notag
\\
& \leq K  E\sum_{k=1}^d \int_{\kappa_n(s)}^s n^{\frac{1}{2}} \big(1+|X_r^{i,N}|+|X_{r}^{i,N, n}|\big)^{\rho/2+1} \big|X_r^{i,N}-X_{r}^{i,N, n}\big| 
\\
& \qquad \times n^{-\frac{1}{2}}  \frac{1}{N} \sum_{j=1}^N \big| \partial_\mu b^{(k)}\big(X_{\kappa_n(s)}^{i,N,n},\mu_{\kappa_n(s)}^{X,N,n}, X_{\kappa_n(s)}^{j,N,n}\big)\big| \Big|\int_{\kappa_n(s)}^{s} \tilde{\sigma}(r, X_{\kappa_n(r)}^{i,N,n}, \mu_{\kappa_n(r)}^{X,N,n})dW_r^i  \Big| dr \notag
\\
& \leq K  E \int_{\kappa_n(s)}^s n \big|X_r^{i,N}-X_{r}^{i,N, n}\big|^2 dr
\\
& \qquad + K  E\sum_{k=1}^d \int_{\kappa_n(s)}^s n^{-1} \big(1+|X_r^{i,N}|+|X_{r}^{i,N, n}|\big)^{\rho+2} 
\\
& \qquad \times \frac{1}{N}\sum_{j=1}^N\big( \big(1+ |X_{\kappa_n(s)}^{i,N,n}|\big)^{\rho+4}+1+|X_{\kappa_n(s)}^{j,N,n}|\big)^2 \Big|\int_{\kappa_n(s)}^{s} \tilde{\sigma}(r, X_{\kappa_n(r)}^{i,N,n}, \mu_{\kappa_n(r)}^{X,N,n})dW_r^i  \Big|^2  dr \notag
\end{align*}
which on using  H\"older's inequality, Corollary \ref{cor:sigma:tilde:bound} and Lemma \ref{lem:moment:bound} gives,
\begin{align} \label{eq:T31}
T_{31} \leq  K  \sup_{r\in[0,s]} E \big|X_r^{i,N}-X_{r}^{i,N, n}\big|^2 + K n^{-3}
\end{align}
for any $s\in[0,T]$ and $n,N\in\mathbb{N}$. 

Moreover, by using Assumption \ref{as:b:lip} and Cauchy-Schwarz inequality, H\"older's inequality and Remark \ref{rem:b:growth},  one obtains, 
\begin{align*}
T_{32}  := &   E\sum_{k=1}^d \int_{\kappa_n(s)}^s \big\{b^{(k)}\big(X_r^{i,N,n}, \mu_r^{X,N}\big)-b^{(k)}\big(X_{r}^{i,N, n}, \mu_{r}^{X,N,n}\big)\big\} dr
\\
& \times \frac{1}{N} \sum_{j=1}^N \Big\langle \partial_\mu b^{(k)}\big(X_{\kappa_n(s)}^{i,N,n},\mu_{\kappa_n(s)}^{X,N,n}, X_{\kappa_n(s)}^{j,N,n}\big), \int_{\kappa_n(s)}^{s} \tilde{\sigma}(r, X_{\kappa_n(r)}^{i,N,n}, \mu_{\kappa_n(r)}^{X,N,n})dW_r^i  \Big \rangle \notag
\\
\leq &   K  E\sum_{k=1}^d \int_{\kappa_n(s)}^s n^{\frac{1}{2}}\mathcal{W}_2\big(\mu_r^{X,N},\mu_{r}^{X,N,n}  \big)  
\\
& \times n^{-\frac{1}{2}} \frac{1}{N} \sum_{j=1}^N \big| \partial_\mu b^{(k)}\big(X_{\kappa_n(s)}^{i,N,n},\mu_{\kappa_n(s)}^{X,N,n}, X_{\kappa_n(s)}^{j,N,n}\big)\big|\Big| \int_{\kappa_n(s)}^{s} \tilde{\sigma}(r, X_{\kappa_n(r)}^{i,N,n}, \mu_{\kappa_n(r)}^{X,N,n})dW_r^i  \Big | dr\notag
\\
\leq &    K  E \int_{\kappa_n(s)}^s n \mathcal{W}_2\big(\mu_r^{X,N},\mu_{r}^{X,N,n}  \big)^2   dr
\\
& +  K  E\sum_{k=1}^d \int_{\kappa_n(s)}^s  n^{-1} \frac{1}{N} \sum_{j=1}^N \big| \partial_\mu b^{(k)}\big(X_{\kappa_n(s)}^{i,N,n},\mu_{\kappa_n(s)}^{X,N,n}, X_{\kappa_n(s)}^{j,N,n}\big)\big|^2 
\\
&  \times \Big| \int_{\kappa_n(s)}^{s} \tilde{\sigma}(r, X_{\kappa_n(r)}^{i,N,n}, \mu_{\kappa_n(r)}^{X,N,n})dW_r^i  \Big |^2  dr\notag
\\
\leq &   K  E\int_{\kappa_n(s)}^s n  \frac{1}{N}\sum_{j=1}^N |X_r^{j,N}-X_r^{j,N,n}|^2    dr +  K  n^{-3} 
\\
\leq & \sup_{i\{1,\ldots<\}} \sup_{r\in[0,s]} E|X_r^{j,N}-X_r^{j,N,n}|^2 +  K  n^{-3}
\end{align*} 
for any $s\in[0,T]$ and $n, N \in\mathbb{N}$. 

To estimate $T_{33}$, we use Cauchy-Schwarz inequality, H\"older's inequality, Lemma~ \ref{lem:b-b}, Remark \ref{rem:b:growth} and Corollary \ref{cor:sigma:tilde:bound} to obtain,
\begin{align*}
T_{33} & :=   E\sum_{k=1}^d \int_{\kappa_n(s)}^s \big\{b^{(k)}\big(X_r^{i,N,n}, \mu_r^{X,N,n}\big)-b_n^{(k)}\big(X_{\kappa_n(r)}^{i,N, n}, \mu_{\kappa_n(r)}^{X,N,n}\big)\big\} dr
\\
& \qquad \times \frac{1}{N} \sum_{j=1}^N \Big\langle \partial_\mu b^{(k)}\big(X_{\kappa_n(s)}^{i,N,n},\mu_{\kappa_n(s)}^{X,N,n}, X_{\kappa_n(s)}^{j,N,n}\big), \int_{\kappa_n(s)}^{s} \tilde{\sigma}(r, X_{\kappa_n(r)}^{i,N,n}, \mu_{\kappa_n(r)}^{X,N,n})dW_r^i  \Big \rangle \notag
\\
&\leq   \sum_{k=1}^d \frac{1}{N} \sum_{j=1}^N \int_{\kappa_n(s)}^s \big\{E\big|b^{(k)}\big(X_r^{i,N,n}, \mu_r^{X,N,n}\big)-b_n^{(k)}\big(X_{\kappa_n(r)}^{i,N, n}, \mu_{\kappa_n(r)}^{X,N,n}\big)\big|^2\big\}^\frac{1}{2} 
\\
& \qquad \times  \Big\{E\big| \partial_\mu b^{(k)}\big(X_{\kappa_n(s)}^{i,N,n},\mu_{\kappa_n(s)}^{X,N,n}, X_{\kappa_n(s)}^{j,N,n}\big)\big|^2 \Big| \int_{\kappa_n(s)}^{s} \tilde{\sigma}(r, X_{\kappa_n(r)}^{i,N,n}, \mu_{\kappa_n(r)}^{X,N,n})dW_r^i  \Big |^2\Big\}^\frac{1}{2} dr 
\\
&\leq K  n^{-2} \notag
\end{align*}
for any $s\in[0,T]$ and $n,N\in\mathbb{N}$. 

 Notice that, 
\begin{align*}
T_{34}& :=   E\sum_{k=1}^d \sum_{l=1}^m \int_{\kappa_n(s)}^s \big\{\sigma^{(k,l)}\big(X_r^{i,N}, \mu_r^{X,N}\big)-\sigma^{(k,l)}\big(X_{r}^{i,N, n}, \mu_{r}^{X,N}\big)\big\} dW_r^{(l),i}
\\
& \qquad \times \frac{1}{N} \sum_{j=1}^N \Big\langle \partial_\mu b^{(k)}\big(X_{\kappa_n(s)}^{i,N,n},\mu_{\kappa_n(s)}^{X,N,n}, X_{\kappa_n(s)}^{j,N,n}\big), \int_{\kappa_n(s)}^{s} \tilde{\sigma}(r, X_{\kappa_n(r)}^{i,N,n}, \mu_{\kappa_n(r)}^{X,N,n})dW_r^i  \Big \rangle \notag
\\
&=  E\sum_{k=1}^d \sum_{l=1}^m \int_{\kappa_n(s)}^s \big\{\sigma^{(k,l)}\big(X_r^{i,N}, \mu_r^{X,N}\big)-\sigma^{(k,l)}\big(X_{r}^{i,N, n}, \mu_{r}^{X,N}\big)\big\} dW_r^{(l),i}
\\
& \qquad \times \frac{1}{N} \sum_{j=1}^N \sum_{l_1=1}^m \int_{\kappa_n(s)}^{s} \Big\langle \partial_\mu b^{(k)}\big(X_{\kappa_n(s)}^{i,N,n},\mu_{\kappa_n(s)}^{X,N,n}, X_{\kappa_n(s)}^{j,N,n}\big), \tilde{\sigma}^{(l_1)}(r, X_{\kappa_n(r)}^{i,N,n}, \mu_{\kappa_n(r)}^{X,N,n})  \Big \rangle dW_r^{(l_1),i} \notag
\\
&=   E\sum_{k=1}^d \sum_{l=1}^m \int_{\kappa_n(s)}^s \big\{\sigma^{(k,l)}\big(X_r^{i,N}, \mu_r^{X,N}\big)-\sigma^{(k,l)}\big(X_{r}^{i,N, n}, \mu_{r}^{X,N}\big)\big\} 
\\
& \qquad \times \frac{1}{N} \sum_{j=1}^N  \Big\langle \partial_\mu b^{(k)}\big(X_{\kappa_n(s)}^{i,N,n},\mu_{\kappa_n(s)}^{X,N,n}, X_{\kappa_n(s)}^{j,N,n}\big), \tilde{\sigma}^{(l)}(r, X_{\kappa_n(r)}^{i,N,n}, \mu_{\kappa_n(r)}^{X,N,n})  \Big \rangle dr \notag
\end{align*}
which on the application of Cauchy-Schwarz inequality, Young's inequality, Remark~\ref{rem:b:growth} and Corollary~\ref{cor:sigma:tilde:bound} gives,
\begin{align}
T_{34}&\leq   E\sum_{k=1}^d \sum_{l=1}^m \int_{\kappa_n(s)}^s \big|\sigma^{(k,l)}\big(X_r^{i,N}, \mu_r^{X,N}\big)-\sigma^{(k,l)}\big(X_{r}^{i,N, n}, \mu_{r}^{X,N}\big)\big| \notag
\\
& \qquad \times \frac{1}{N} \sum_{j=1}^N  \big| \partial_\mu b^{(k)}\big(X_{\kappa_n(s)}^{i,N,n},\mu_{\kappa_n(s)}^{X,N,n}, X_{\kappa_n(s)}^{j,N,n}\big)\big| \big| \tilde{\sigma}^{(l)}(r, X_{\kappa_n(r)}^{i,N,n}, \mu_{\kappa_n(r)}^{X,N,n})  \big| dr \notag
\\
& \leq K  E\sum_{k=1}^d \sum_{l=1}^m \int_{\kappa_n(s)}^s n^{\frac{1}{2}}\big|X_r^{i,N}-X_{r}^{i,N, n}| \notag
\\
& \qquad \times n^{-\frac{1}{2}} \frac{1}{N} \sum_{j=1}^N  \big| \partial_\mu b^{(k)}\big(X_{\kappa_n(s)}^{i,N,n},\mu_{\kappa_n(s)}^{X,N,n}, X_{\kappa_n(s)}^{j,N,n}\big)\big| \big| \tilde{\sigma}^{(l)}(r, X_{\kappa_n(r)}^{i,N,n}, \mu_{\kappa_n(r)}^{X,N,n})  \big| dr \notag
\\
& \leq K  E\int_{\kappa_n(s)}^s n\big|X_r^{i,N}-X_{r}^{i,N, n}|^2  dr \notag
\\
& \qquad + K E\sum_{k=1}^d \sum_{l=1}^m \int_{\kappa_n(s)}^s    \frac{n^{-1}}{N}  \sum_{j=1}^N  \big| \partial_\mu b^{(k)}\big(X_{\kappa_n(s)}^{i,N,n},\mu_{\kappa_n(s)}^{X,N,n}, X_{\kappa_n(s)}^{j,N,n}\big)\big|^2  \big| \tilde{\sigma}^{(l)}(r, X_{\kappa_n(r)}^{i,N,n}, \mu_{\kappa_n(r)}^{X,N,n})  \big|^2 dr \notag
\\
 & \leq K  \sup_{i\in \{1,\ldots, N\}}\sup_{r\in[0,s]}E\big|X_r^{i,N}-X_{r}^{i,N, n}|^2  + K n^{-2} \label{eq:T34}
\end{align}
for any $s\in[0,T]$ and $n,N\in\mathbb{N}$. 

Also notice that, 
\begin{align*}
T_{35}& :=  E\sum_{k=1}^d \sum_{l=1}^m \int_{\kappa_n(s)}^s \big\{\sigma^{(k,l)}\big(X_r^{i,N,n}, \mu_r^{X,N}\big)-\sigma^{(k,l)}\big(X_{r}^{i,N, n}, \mu_{r}^{X,N,n}\big)\big\} dW_r^{(l),i}
\\
& \qquad \times \frac{1}{N} \sum_{j=1}^N \Big\langle \partial_\mu b^{(k)}\big(X_{\kappa_n(s)}^{i,N,n},\mu_{\kappa_n(s)}^{X,N,n}, X_{\kappa_n(s)}^{j,N,n}\big), \int_{\kappa_n(s)}^{s} \tilde{\sigma}(r, X_{\kappa_n(r)}^{i,N,n}, \mu_{\kappa_n(r)}^{X,N,n})dW_r^i  \Big \rangle \notag
\\
& =   E\sum_{k=1}^d \sum_{l=1}^m \int_{\kappa_n(s)}^s \big\{\sigma^{(k,l)}\big(X_r^{i,N,n}, \mu_r^{X,N}\big)-\sigma^{(k,l)}\big(X_{r}^{i,N, n}, \mu_{r}^{X,N,n}\big)\big\} dW_r^{(l),i}
\\
& \qquad \times \frac{1}{N} \sum_{j=1}^N \sum_{l_1=1}^m \Big\langle \partial_\mu b^{(k)}\big(X_{\kappa_n(s)}^{i,N,n},\mu_{\kappa_n(s)}^{X,N,n}, X_{\kappa_n(s)}^{j,N,n}\big), \int_{\kappa_n(s)}^{s} \tilde{\sigma}^{(l_1)}(r, X_{\kappa_n(r)}^{i,N,n}, \mu_{\kappa_n(r)}^{X,N,n}) \Big \rangle dW_r^{(l_1),i } \notag
\\
& =   E\sum_{k=1}^d \sum_{l=1}^m \int_{\kappa_n(s)}^s \big\{\sigma^{(k,l)}\big(X_r^{i,N,n}, \mu_r^{X,N}\big)-\sigma^{(k,l)}\big(X_{r}^{i,N, n}, \mu_{r}^{X,N,n}\big)\big\} 
\\
& \qquad \times \frac{1}{N} \sum_{j=1}^N  \Big\langle \partial_\mu b^{(k)}\big(X_{\kappa_n(s)}^{i,N,n},\mu_{\kappa_n(s)}^{X,N,n}, X_{\kappa_n(s)}^{j,N,n}\big), \int_{\kappa_n(s)}^{s} \tilde{\sigma}^{(l)}(r, X_{\kappa_n(r)}^{i,N,n}, \mu_{\kappa_n(r)}^{X,N,n}) \Big \rangle dr \notag
\end{align*}
and then one uses Assumption \ref{as:sig:lip}, Remark \ref{rem:b:growth}, Cauchy-Schwarz inequality, Young's inequality, H\"older's inequality, Corollary \ref{cor:sigma:tilde:bound} and Lemma \ref{lem:moment:bound} to obtain,
\begin{align}
T_{35} & \leq K  E\sum_{k=1}^d \sum_{l=1}^m \int_{\kappa_n(s)}^s n^\frac{1}{2} \mathcal{W}_2\big(\mu_r^{X,N}, \mu_{r}^{X,N,n} \big) \notag
\\
& \qquad \times n^{-\frac{1}{2}} \frac{1}{N} \sum_{j=1}^N  \big| \partial_\mu b^{(k)}\big(X_{\kappa_n(s)}^{i,N,n},\mu_{\kappa_n(s)}^{X,N,n}, X_{\kappa_n(s)}^{j,N,n}\big)\big|\Big| \int_{\kappa_n(s)}^{s} \tilde{\sigma}^{(l)}(r, X_{\kappa_n(r)}^{i,N,n}, \mu_{\kappa_n(r)}^{X,N,n}) \Big | dr \notag
 \\
 & \leq K  E\int_{\kappa_n(s)}^s n \frac{1}{N} \sum_{j=1}^N |X_r^{i,N}-X_r^{i,N,n}|^2 dr \notag
\\
& \qquad + K  \sum_{k=1}^d \sum_{l=1}^m \int_{\kappa_n(s)}^s  n^{-1} \frac{1}{N^2} N \sum_{j=1}^N  \big\{E\big| \partial_\mu b^{(k)}\big(X_{\kappa_n(s)}^{i,N,n},\mu_{\kappa_n(s)}^{X,N,n}, X_{\kappa_n(s)}^{j,N,n}\big)\big|^4 \big\}^\frac{1}{4} \notag 
\\
& \qquad \times \Big\{E\Big| \int_{\kappa_n(s)}^{s} \tilde{\sigma}^{(l)}(r, X_{\kappa_n(r)}^{i,N,n}, \mu_{\kappa_n(r)}^{X,N,n}) \Big |^4 \Big\}^\frac{1}{4} dr \notag
\\
& \leq K \sup_{i\in\{1,\cdots,N\}}\sup_{r\in[0,s]} E|X_r^{i,N}-X_r^{i,N,n}|^2 + K n^{-2} \label{eq:T35}
\end{align} 
for any $s\in[0,T]$ and $n,N\in\mathbb{N}$.  

Similarly, 
\begin{align*}
T_{36}&:=  E\sum_{k=1}^d \sum_{l=1}^m \int_{\kappa_n(s)}^s \big\{\sigma^{(k,l)}\big(X_r^{i,N,n}, \mu_r^{X,N,n}\big)-\tilde{\sigma}^{(k,l)}\big(r, X_{\kappa_n(r)}^{i,N, n}, \mu_{\kappa_n(r)}^{X,N,n}\big)\big\} dW_r^{(l),i}
\\
& \qquad \times \frac{1}{N} \sum_{j=1}^N \Big\langle \partial_\mu b^{(k)}\big(X_{\kappa_n(s)}^{i,N,n},\mu_{\kappa_n(s)}^{X,N,n}, X_{\kappa_n(s)}^{j,N,n}\big), \int_{\kappa_n(s)}^{s} \tilde{\sigma}(r, X_{\kappa_n(r)}^{i,N,n}, \mu_{\kappa_n(r)}^{X,N,n})dW_r^i  \Big \rangle \notag
\\
& = E\sum_{k=1}^d \sum_{l=1}^m \int_{\kappa_n(s)}^s \big\{\sigma^{(k,l)}\big(X_r^{i,N,n}, \mu_r^{X,N,n}\big)-\tilde{\sigma}^{(k,l)}\big(r, X_{\kappa_n(r)}^{i,N, n}, \mu_{\kappa_n(r)}^{X,N,n}\big)\big\} dW_r^{(l),i}
\\
& \qquad \times \frac{1}{N} \sum_{j=1}^N \sum_{l_1=1}^m \int_{\kappa_n(s)}^{s} \Big\langle \partial_\mu b^{(k)}\big(X_{\kappa_n(s)}^{i,N,n},\mu_{\kappa_n(s)}^{X,N,n}, X_{\kappa_n(s)}^{j,N,n}\big),  \tilde{\sigma}^{(l_1)}(r, X_{\kappa_n(r)}^{i,N,n}, \mu_{\kappa_n(r)}^{X,N,n})\Big \rangle dW_r^{(l_1),i}   \notag
\\
& =  E\sum_{k=1}^d \sum_{l=1}^m \int_{\kappa_n(s)}^s \big\{\sigma^{(k,l)}\big(X_r^{i,N,n}, \mu_r^{X,N,n}\big)-\tilde{\sigma}^{(k,l)}\big(r, X_{\kappa_n(r)}^{i,N, n}, \mu_{\kappa_n(r)}^{X,N,n}\big)\big\} 
\\
& \qquad \times \frac{1}{N} \sum_{j=1}^N  \Big\langle \partial_\mu b^{(k)}\big(X_{\kappa_n(s)}^{i,N,n},\mu_{\kappa_n(s)}^{X,N,n}, X_{\kappa_n(s)}^{j,N,n}\big),  \tilde{\sigma}^{(l)}(r, X_{\kappa_n(r)}^{i,N,n}, \mu_{\kappa_n(r)}^{X,N,n})\Big \rangle dr   \notag
\end{align*}
and then using H\"older's inequality, Cauchy-Schwarz inequality, Remark \ref{rem:b:growth},  Lemma~\ref{lem:sigma:rate} and Corollary  \ref{cor:sigma:tilde:bound}, one obtains,
\begin{align}
T_{36} & \leq   E\sum_{k=1}^d \sum_{l=1}^m \int_{\kappa_n(s)}^s \big|\sigma^{(k,l)}\big(X_r^{i,N,n}, \mu_r^{X,N,n}\big)-\tilde{\sigma}^{(k,l)}\big(r, X_{\kappa_n(r)}^{i,N, n}, \mu_{\kappa_n(r)}^{X,N,n}\big)\big| \notag
\\
& \qquad \times \frac{1}{N} \sum_{j=1}^N  \big| \partial_\mu b^{(k)}\big(X_{\kappa_n(s)}^{i,N,n},\mu_{\kappa_n(s)}^{X,N,n}, X_{\kappa_n(s)}^{j,N,n}\big)\big |  \big |   \tilde{\sigma}^{(l)}(r, X_{\kappa_n(r)}^{i,N,n}, \mu_{\kappa_n(r)}^{X,N,n})\big | dr   \notag
\\
& \leq   \sum_{k=1}^d \sum_{l=1}^m \int_{\kappa_n(s)}^s \big\{E \big|\sigma^{(k,l)}\big(X_r^{i,N,n}, \mu_r^{X,N,n}\big)-\tilde{\sigma}^{(k,l)}\big(r, X_{\kappa_n(r)}^{i,N, n}, \mu_{\kappa_n(r)}^{X,N,n}\big)\big|^2\big\}^\frac{1}{2} \notag
\\
& \qquad \times \frac{1}{N} \sum_{j=1}^N  \big\{E\big| \partial_\mu b^{(k)}\big(X_{\kappa_n(s)}^{i,N,n},\mu_{\kappa_n(s)}^{X,N,n}, X_{\kappa_n(s)}^{j,N,n}\big)\big |^4\big\}^\frac{1}{4} \big\{ E\big |   \tilde{\sigma}^{(l)}(r, X_{\kappa_n(r)}^{i,N,n}, \mu_{\kappa_n(r)}^{X,N,n})\big |^4\big\}^\frac{1}{4} dr   \notag
\\
&\leq K n^{-2} \label{eq:T36}
\end{align}
for any $s\in[0,T]$ and $n,N\in\mathbb{N}$. Substituting values from \eqref{eq:T31} to \eqref{eq:T36} in equation \eqref{eq:T3}, we get 
\begin{align} \label{eq:T3:final}
T_3 \leq  K \sup_{i\in\{1,\cdots,N\}}\sup_{r\in[0,s]} E|X_r^{i,N}-X_r^{i,N,n}|^2 + K n^{-2}
\end{align}
for any $s\in[0,T]$ and $n,N\in\mathbb{N}$. Substituting estimates from \eqref{eq:T1}, \eqref{eq:T2:final} and \eqref{eq:T3:final} in equation \eqref{eq:T} completes the proof. 
\end{proof}
\subsection*{Proof of Proposition \ref{prop:milstein}} Recall equations \eqref{eq:interactint:particle} and \eqref{eq:scheme} and use It\^o's formula to obtain, 
\begin{align*}
\big|X_t^{i,N} & -X_t^{i,N, n}\big|^2  = 2  \int_0^t \big\langle X_s^{i,N}-X_s^{i,N, n}, b\big(X_s^{i,N}, \mu_s^{X,N}\big)-b_n\big(X_{\kappa_n(s)}^{i,N, n}, \mu_{\kappa_n(s)}^{X,N,n}\big)\big\rangle ds
\\
& +  \int_0^t \big\langle X_s^{i,N}-X_s^{i,N, n}, \big\{\sigma\big(X_s^{i,N}, \mu_s^{X,N}\big)-\tilde{\sigma}\big(s,X_{\kappa_n(s)}^{i,N,n}, \mu_{\kappa_n(s)}^{X,N,n}\big)  \big\}dW_s^i \big\rangle 
\\
& +  \int_0^t \big|\sigma\big(X_s^{i,N}, \mu_s^{X,N}\big)-\tilde{\sigma}\big(s,X_{\kappa_n(s)}^{i,N,n}, \mu_{\kappa_n(s)}^{X,N,n}\big)  \big|^2 ds
\end{align*}
almost surely for any $t\in[0,T]$ and $n,N\in\mathbb{N}$. Hence, 
\begin{align*}
E\big|X_t^{i,N} & -X_t^{i,N, n}\big|^2  \leq  2  E \int_0^t \big\langle X_s^{i,N}-X_s^{i,N, n}, b\big(X_s^{i,N}, \mu_s^{X,N}\big)-b\big(X_{s}^{i,N, n},\mu_s^{X,N}\big)\big\rangle ds
\\
& + 2  E \int_0^t \big\langle X_s^{i,N}-X_s^{i,N, n}, b\big(X_s^{i,N,n}, \mu_s^{X,N}\big)-b\big(X_s^{i,N,n}, \mu_s^{X,N,n}\big)\big\rangle ds
\\
& + 2  E \int_0^t \big\langle X_s^{i,N}-X_s^{i,N, n}, b\big(X_s^{i,N,n}, \mu_s^{X,N,n}\big)-b_n\big(X_{\kappa_n(s)}^{i,N, n}, \mu_{\kappa_n(s)}^{X,N,n}\big) \big\rangle ds
\\
& +K  E \int_0^t \big|\sigma\big(X_s^{i,N}, \mu_s^{X,N}\big)-\sigma\big(X_s^{i,N,n}, \mu_s^{X,N}\big) \big|^2 ds
\\
& + K E \int_0^t \big|\sigma\big(X_s^{i,N,n}, \mu_s^{X,N}\big)-\sigma\big(X_s^{i,N,n}, \mu_s^{X,N,n}\big)  \big|^2 ds
\\
& + K E \int_0^t \big|\sigma\big(X_s^{i,N,n}, \mu_s^{X,N,n}\big)-\tilde{\sigma}\big(s,X_{\kappa_n(s)}^{i,N,n}, \mu_{\kappa_n(s)}^{X,N,n}\big)  \big|^2 ds
\end{align*}
which on using Assumption \ref{as:b:lip}, Assumption \ref{as:sig:lip}, Young's inequality, Lemma \ref{lem:sigma:rate} and Lemma \ref{lem:b-b:x} yields,
\begin{align*}
E\big|X_t^{i,N} & -X_t^{i,N, n}\big|^2  \leq  K   n^{-2}  + E \int_0^t \big| X_s^{i,N}-X_s^{i,N, n}\big|^2 ds 
\\
& \qquad + K  E \int_0^t \mathcal{W}_2\big( \mu_s^{X,N}, \mu_s^{X,N,n}\big)^2  ds 
\end{align*}
which further implies, 
\begin{align*}
\sup_{i\in\{1,\ldots,N\}}\sup_{r\in[0,t]}E\big|X_r^{i,N} & -X_r^{i,N, n}\big|^2 \leq K n^{-2}+\int_{0}^t \sup_{i\in\{1,\ldots,N\}}\sup_{r\in[0,s]}E\big|X_r^{i,N} & -X_r^{i,N, n}\big|^2 ds
\end{align*}
for any $t\in[0,T]$ and $n,N\in\mathbb{N}$. Thus, the Gronwall's inequality completes the proof. 

\end{document}